\documentclass[11pt]{article}
\pdfoutput=1
\usepackage[pdftex]{graphicx}
\usepackage{ifpdf,psfig,subfigure,verbatim,indentfirst,cleveref}
\usepackage{amsmath,amssymb,amsthm}

\usepackage{amsmath}
\usepackage{cite}
\usepackage{graphicx}
 \usepackage{color}
 \newtheorem{theorem}{Theorem}[section]
\newtheorem{lemma}[theorem]{Lemma}

\newtheorem{corollary}[theorem]{Corollary}

\theoremstyle{definition}

\newtheorem{example}[theorem]{Example}

\theoremstyle{remark}
\newtheorem{remark}[theorem]{Remark}

\numberwithin{equation}{section}

\def\be {\begin{equation}}
\def\ee {\end{equation}}
\def\ba {\begin{eqnarray}}
\def\ea {\end{eqnarray}}

\begin{document}

\title{\Large\bf{Turing instability and Turing-Hopf bifurcation in diffusive Schnakenberg systems with gene expression time delay
}}
\author{{Weihua Jiang\thanks {Corresponding author. E-mail address: jiangwh@hit.edu.cn}}, Hongbin Wang, Xun Cao\\
\footnotesize {\em Department of Mathematics, Harbin Institute of Technology, Harbin 150001,China }}

\date{}
\maketitle

\baselineskip=0.9\normalbaselineskip \vspace{-3pt}

\begin{abstract}
For delayed reaction-diffusion Schnakenberg systems with Neumann boundary conditions, critical conditions for Turing instability are derived, which are necessary and sufficient.
And existence conditions for Turing, Hopf and Turing-Hopf bifurcations are established. Normal forms truncated to order 3 at Turing-Hopf singularity of codimension 2, are derived.
By investigating Turing-Hopf bifurcation, the parameter regions for the stability of a periodic solution, a pair of spatially inhomogeneous steady states and a pair of spatially inhomogeneous periodic solutions, are derived in $(\tau,\varepsilon)$ parameter plane ($\tau$ for time delay, $\varepsilon$ for diffusion rate). It is revealed that joint effects of diffusion and delay can lead to the occurrence of mixed spatial and temporal patterns.
Moreover, it is also demonstrated that various spatially inhomogeneous patterns with different spatial frequencies can be achieved via changing the diffusion rate.
And, the phenomenon that time delay may induce a failure of Turing instability observed by Gaffney and Monk (2006) are theoretically explained.

\end{abstract}

\section{Introduction}

A morphogen is an important concept in developmental biology, because it describes a mechanism by which the emission of a signal from one part of an embryo can determine location, differentiation and fate of many surrounding cells \cite{Gurdon2001Morphogen}. Schnakenberg system \cite{Schnakenberg1979Simple} has been used to model spatial distribution of morphogen, and to understand how various morphogens interact with cells and patterns \cite{Arafa2012Approximate,Ward2002The}. Although Schnakenberg system has a simple structure, it is one of the few reaction-diffusion models in morphogenesis, which exhibit patterns consistent with those in experiments \cite{Ricard2009Turing}.
%

In the context of cellular pattern formation, delays play a central role in the generation of spatially coordinated oscillations of gene expression underlying the formation of vertebrate somites \cite{Lewis2003Autoinhibition}. Delays are believed to have a profound effects on the mode and tempo of cellular pattern formation \cite{Veflingstad2005Effect}.
Considering that activator autocatalysis in reaction-diffusion mechanism occurs via gene expression, Gaffney and Monk introduced gene expression time delay which is induced by transcription and translation, into Schnakenberg reaction-diffusion equations. They found that time delay educe a failure of Turing instability, which can't be predicted by a naive linear analysis of the underlying equations about homogeneous steady states (see\cite{Gaffney2006Gene}). On the basis of above work, Yi, Gaffney and Seirin-Lee considered following delayed reaction-diffusion Schnakenberg system incorporating gene expression delays, under Neumann boundary conditions
\begin{equation}\label{3.1b}
 \begin{array}{rll}
 u_t(x,t)=&\varepsilon d u_{xx}(x,t)+a-u(x,t)+u^2(x,t-\tau)v(x,t-\tau), &x\in(0,1),t>0,\\
 v_t(x,t)=&d v_{xx}(x,t)+b-u^2(x,t-\tau)v(x,t-\tau), &x\in(0,1),t>0,\\
 u_x(0,t)=&u_x(1,t)=v_x(0,t)=v_x(1,t)=0, \;\; &t\geq0,\\
 u(x,t)=&\phi(x,t)\geq0, v(x,t)=\varphi(x,t)\geq0,&
(x,t)\in[0,1]\times[-\tau,0],
  \end{array}
\end{equation}
where $u(x,t)$ and $v(x,t)$ are concentrations of activator and inhibitor
at $(x, t)$ respectively, and $a,b,d, \varepsilon$ are all positive constants, see  \cite{YGLM}. The detailed model derivation can be found elsewhere \cite{SeirinLee2010} and \cite{Gaffney2006Gene}.

Yi, Gaffney and Seirin-Lee performed detailed stability and Hopf bifurcation analyses, and derived conditions for determining the direction of bifurcation and stability of bifurcating periodic solution. Diffusion-driven instability of the unique spatially inhomogeneous steady state solution and delay-driven instability of spatially homogeneous periodic solution were also investigated, see \cite{YGLM}.


Turing's theory \cite{Turing1952} shows that diffusion could destabilize an otherwise stable equilibrium of reaction-diffusion equations, and lead to nonuniform spatial patterns. This kind of instability is usually called Turing instability or diffusion-driven instability. This corresponds to the spontaneous formation of a spatially inhomogeneous state in a Turing bifurcation, see \cite{Jang2004Global, Ni2005Turing, Murray2015, Li2013Hopf}.
Due to the time-delay factor, Hopf bifurcations occur more frequently in delayed differential equations, which could destabilize a stable equilibrium and lead to temporally inhomogeneous patterns, see \cite{Hale1977,Yi2009,Su2010,YuanXP2010,ChenS2012,YiFQ2013,Guo2015,Chen2016Stability,Wang2016Spatiotemporal}.
In presence of diffusion and time delay, Turing-Hopf bifurcation arises extensively from the coincidence of Turing bifurcation and Hopf bifurcation. Thus, complex spatiotemporal behaviors involving dynamical interactions of two Fourier modes, which has both nonuniform spatially and temporally periodic patterns, can be found by investigating Turing-Hopf bifurcation, see \cite{Hadeler2012Interaction,Just2001Spatiotemporal,Kidachi1980On,Holmes1997,Ruan2015,SongY2016} and references therein.
The interaction of Turing and Hopf bifurcations are regarded as an important mechanism for the appearance of complex spatiotemporal dynamics in diffusive models.

In the present paper, we concentrate on Turing instability and Turing-Hopf bifurcation of system \eqref{3.1b}. The main work is as follows:

($\bf{1}$) On the basis of results of \cite{YGLM}, we have obtained a much larger range where Turing instability does not occur, which is one sufficient and necessary condition.  In other words, we give the weaker conditions that guarantee Turing instability. Meanwhile, the maximum parameter region, where the coexistence equilibrium is stable, is provided, of which the boundary consists of Turing bifurcation curves.

($\bf{2}$) We have given an explicit expression for the first Turing bifurcation curve, on which corresponding characteristic equations without delay have no root with positive real part. It is a piecewise smooth and continuous curve, and the piecewise points are exactly Turing-Turing bifurcation points. The expression explicitly depends on wave numbers and diffusion coefficients, by which we will easily find spatial patterns with arbitrary wave number.

($\bf{3}$) The joint effects of diffusion and delay ensure that Turing-Hopf bifurcation takes place. Within the framework of Faria \cite{Faria2000,Faria2002}, Jiang et al.\cite{JAS} gave explicitly generic formulas for calculating coefficients of normal forms up to order 3 for codimension-two Hopf-steady state bifurcation of delayed reaction-diffusion equations with Neumann boundary conditions, which include Turing-Hopf bifurcation.
Based on above work, normal forms truncated to order 3 for \eqref{3.1b} are established. All coefficients of normal forms are expressed explicitly, utilizing original system parameters $a,b,d,\varepsilon$ and delay $\tau$. Based on this, it is very convenient to analyze and draw the impact of original system parameters on dynamical behaviors.

($\bf{4}$) By analyzing Turring-Hopf bifurcation, we demonstrated that delay can destabilize a stable equilibrium and drive diffusive system to generate a stable temporally periodic orbit via Hopf bifurcation,
and diffusion can also destabilize the stable equilibrium and lead to a pair of stable spatially inhomogeneous steady state solutions.
The joint effects of diffusion and delay can destabilize above stable temporally periodic orbit and a pair of  spatially inhomogeneous steady state solutions, and induce a pair of stable spatially inhomogeneous periodic orbits, which are mixed spatiotemporal periodic patterns. And various spatially inhomogeneous patterns  with different spatial frequencies , can be achieved via changing the diffusion rates.

The paper is organized as follows. In Section 2, by analyzing characteristic equations at the coexistence equilibrium, conditions for Turing instability, as well as existence conditions of Turing bifurcation, Hopf bifurcation and Turing-Hopf bifurcation are established.
In Section 3, applying generic formulas developed by Jiang et al.\cite{JAS}, explicit formulas for
quadratic and cubic coefficients of normal forms at Turing-Hopf singularity are derived.
In Section 4, we take a set of system parameters to illustrate spatiotemporal patterns with Turing-Hopf bifurcation, by analyzing normal form derived in Section 3. And numerical simulations demonstrate spatiotemporal phenomena consistent with theoretical analyses. We finish our study with some conclusions in Section 5.

%
%
\section{Turing instability, Turing bifurcation and Hopf bifurcation}
In this section, we would like  to investigate Turing instability, Turing bifurcation and Hopf bifurcation for delayed reaction-diffusion Schnakenberg system \eqref{3.1b}.
The unique positive constant steady state solution is denoted by $E_*=(u_*,v_*)$, where
\begin{equation}\label{fixp}
u_*=a+b,\;v_*=\frac{b}{(a+b)^2}.
\end{equation}
The linearized equations of system (\ref{3.1b}) evaluated at $(u_*,v_*)$ are given by:
\begin{equation}\label{3.2bbb}
 \begin{array}{rll}
  u_t(x,t)=&\varepsilon du_{xx}(x,t)-u(x,t)+2u_*v_*u(x,t-\tau)+u_*^2v(x,t-\tau), &x\in(0,1),t>0,\\
 v_t(x,t)=&d v_{xx}(x,t)-2u_*v_*u(x,t-\tau)-u_*^2v(x,t-\tau), &x\in(0,1),t>0,\\
 u_x(0,t)=&u_x(1,t)=v_x(0,t)=v_x(1,t)=0, \;\; &t\geq0.
  \end{array}
\end{equation}

Recall that $\mu_k=k^2\pi^2$ with $k\in\mathbb{N}_0$ are eigenvalues of $-\Delta$ in one dimensional spatial domain $(0,1)$. Then, a straightforward analysis indicates that the eigenvalues of linearized operator can be derived by discussing roots of following series of equations,
\begin{equation}\label{3.3k}
D_k(\lambda,\tau,\varepsilon):=\mathrm{det}\left(\begin{array}{cc} -\varepsilon d
k^2\pi^2-1+2u_*v_*e^{-\lambda\tau}-\lambda
&  u_*^2e^{-\lambda\tau}\\
-2u_*v_*e^{-\lambda\tau} & -dk^2\pi^2- u_*^2e^{-\lambda\tau}-\lambda
\end{array}\right)=0,\;k\in\mathbb{N}_0,
\end{equation}
or equivalently,
\begin{equation}\label{eigen}
D_k(\lambda,\tau,\varepsilon):=\lambda^2+p_k\lambda+r_k+(s_k\lambda+q_k)e^{-\lambda\tau}=0,\;k\in\mathbb{N}_0,
\end{equation}
where
\begin{equation}\label{cf}
\begin{aligned}
p_k=&p_k(\varepsilon)\triangleq (\varepsilon+1)dk^2\pi^2+1,\;&r_k=&r_k(\varepsilon)\triangleq \varepsilon d^2k^4\pi^4+dk^2\pi^2,\\
s_k\triangleq & u_*^2-2u_*v_*,\;&q_k=&q_k(\varepsilon)\triangleq (\varepsilon u_*^2-2u_*v_*)dk^2\pi^2+u_*^2.
\end{aligned}
\end{equation}
Therefore,
\begin{equation}\label{eigen-0}
D_k(\lambda,0,\varepsilon)=\lambda^2+(p_k+s_k)\lambda+(r_k+q_k)=0,\;k\in\mathbb{N}_0.
\end{equation}
We define $DET_k= r_k+q_k,~~TR_k=  -(p_k+s_k)$, then
$$ \begin{array}{rl}
DET_k=&\varepsilon d^2k^4\pi^4+(\varepsilon u_*^2-2u_*v_*+1)dk^2\pi^2+u_*^2,\\
TR_k=&-(\varepsilon+1)dk^2\pi^2-1-u_*^2+2u_*v_*,
\end{array}$$
for $k\in\mathbb{N}_0$. We firstly consider steady state bifurcation.

{\noindent\bf (1) Turing instability and diffusion-induced a steady state bifurcation}

We assume that
\begin{enumerate}
\item[${\bf (N_0)}$] $ u_*^2>2u_*v_*-1>0$.
\end{enumerate}
By ${\bf (N_0)}$,  all eigenvalues of $D_0(\lambda,0)$ have negative real parts, and $TR_k<0$  for $k\in\mathbb{N}_0$. Now, we consider what conditions cause Turing instability. Assume that
\begin{enumerate}
\item[${\bf (N_1)}$] $0<\varepsilon<\varepsilon_1,~\mathrm{where} ~ \varepsilon_1\triangleq\frac{1}{u_*^2}(\sqrt{2u_*v_*}-1)^2$.

\item[${\bf (N_2)}$] $0<\varepsilon< \varepsilon_2(d),~\mathrm{where}~ \varepsilon_2(d)\triangleq\frac{2u_*v_*-1}{\pi^2d+u_*^2},~d>0$.
\end{enumerate}
 Let $k_{min}^2$ be the minimal point of function $DET_k$ on $k^2 \in \mathbb{R}_+$, then, in $\mathbb{R}_+$,
$$k_{min}=\sqrt{\frac{1}{2 d}\frac{2u_*v_*-\varepsilon u_*^2-1}{\varepsilon \pi^2}}.$$
Condition ${\bf (N_1)}$ guarantees that $\min\limits_{k\in \mathbb{R}_+ } DET_k < 0$ and $\varepsilon u_*^2<2u_*v_*-1$. Combining condition ${\bf (N_0)}$, we have $\varepsilon<1$. Moreover, condition  ${\bf (N_2)}$ guarantees that the minimal point
$$k_{min}>\frac{1}{\sqrt{2}}.$$
We know that $\varepsilon=\varepsilon_2(d)$ decreases monotonically  in $d$  and intersects with $\varepsilon=\varepsilon_1$ at the point $d=d_0$, where $d_0\triangleq\frac{2u_*^2}{\pi^2(\sqrt{2u_*v_*}-1)}$. We take $\varepsilon_B(d)=\min\limits_{d>0}\{\varepsilon_1,\varepsilon_2(d)\}$, then
\begin{equation}\label{eB}
\varepsilon_B(d)=\left\{\begin{array}{lll} &\varepsilon_1, &\mathrm{if} ~0<d\leq d_0,\\
&\varepsilon_2(d), &\mathrm{if} ~~~~~~~d\geq d_0.
\end{array}\right.
\end{equation}
Thus we have following conclusions.

\begin{lemma}\label{l21}
Suppose that ${\bf (N_0)}$ holds, then assumptions ${\bf (N_1)}$ and ${\bf (N_2)}$ hold  if and only if   $$0<\varepsilon< \varepsilon_B(d),~d>0.$$
\end{lemma}
Denote
\begin{equation}\label{e0}
\varepsilon_*(k,d)=\frac{(2u_*v_*-1)dk^2\pi^2-u_*^2}{dk^2\pi^2(dk^2\pi^2+u_*^2)},
~\mathrm{for}~d>d_{k},~k\in\mathbb{N},
\end{equation}
where $d_{k}\triangleq\frac{u_*^2}{2u_*v_*-1}\frac{1}{k^2\pi^2}$,
then $DET_{k}=0$ when $\varepsilon=\varepsilon_*(k,d)$.

\begin{lemma}\label{l22}
 Suppose that ${\bf (N_0)}$ holds, function $\varepsilon=\varepsilon_*(k,d)$ has following properties,
\begin{enumerate}
\item[(a)] $\varepsilon=\varepsilon_*(k,d)$ will reach the maximum $\varepsilon_1$ at $d=d_M(k)$, and $d=d_M(k)$ decreases monotonically  as $k$ increases, where $d_M(k)\triangleq \frac{u_*^2}{\sqrt{2u_*v_*}-1}\frac{1}{k^2\pi^2}$,  $k\in\mathbb{N}$.

\item[(b)] $\varepsilon=\varepsilon_*(k,d)$ decreases monotonically  in $d$ and $k$ when $d>d_{M}(k)$.

\item[(c)] For $k\in\mathbb{N}$, following equation
$$\varepsilon_*(k,d)=\varepsilon_*(k+1,d), d>0,$$
only has one root $d_{k,k+1} ~\in (d_M(k+1),d_M(k))$, where
\begin{equation}\label{dkk+}
d_{k,k+1}=\frac{u_*^2}{2\pi^2(2u_*v_*-1)}\left[{k^2}+
\frac{1}{(k+1)^2}+\sqrt{\left(\frac{1}{k^2}+
\frac{1}{(k+1)^2}\right)^2+\frac{4(2u_*v_*-1)}{k^2(k+1)^2}}
\right].
\end{equation}
And
$$\varepsilon_*(k,d)>\varepsilon_*(k+1,d)>\varepsilon_*(k+2,d)>\cdots, ~\mathrm{for}~d>d_{k,k+1}.$$

\item[(d)] Mark $d_{0,1}=+\infty$ and denote
\begin{equation}\label{ed}
\varepsilon_*\triangleq\varepsilon_*(d)=\varepsilon_*(k,d),~d\in [d_{k,k+1},d_{k-1,k}),~~k\in\mathbb{N},
\end{equation}
then $$\varepsilon_*(d)\leq\varepsilon_B(d),~0<d<+\infty,$$
and $\varepsilon_*(d)=\varepsilon_B(d)$ if and only if $d=d_M(k),~k\in\mathbb{N}.$
\end{enumerate}
\end{lemma}
In Figure \ref{TuringF-1}, we present a graph  of functions $\varepsilon=\varepsilon_1,~\varepsilon=\varepsilon_2(d),~\mathrm{and}~\varepsilon=\varepsilon_*(k,d),~d>0,~k=1,2,3\cdots$, which will help us understand the results of Lemma \ref{l21} and \ref{l22}.

 \begin{figure}[htbp]
	\centering
			\includegraphics[scale=0.60]{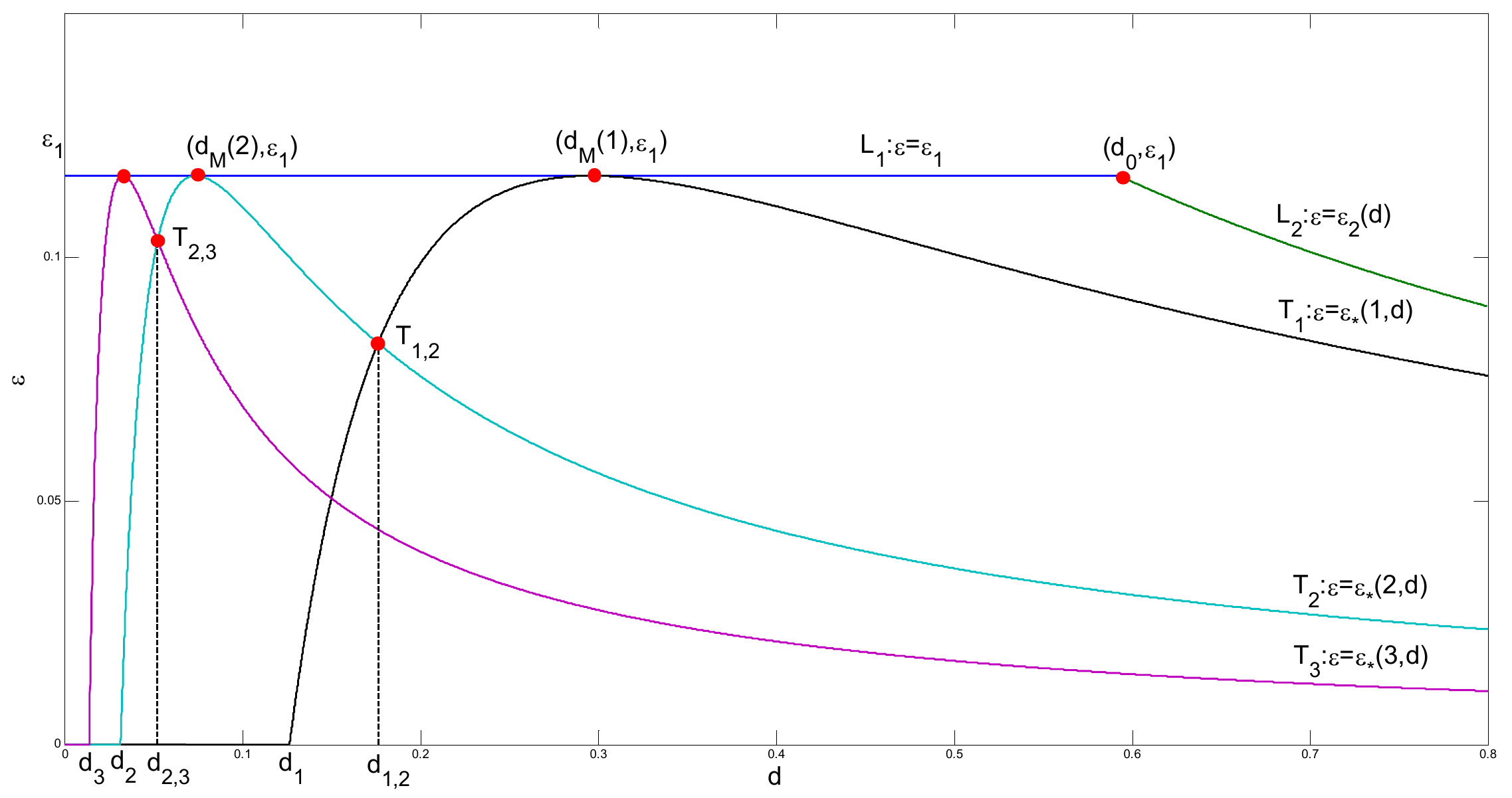}
	\caption{The graph of functions $\varepsilon=\varepsilon_1,~\varepsilon=\varepsilon_2(d),~\mathrm{and}~\varepsilon=\varepsilon_*(k,d),~d>0,~k=1,2,3\cdots$,  in $(d,\varepsilon)$ plane.}\label{TuringF-1}
\end{figure}

\begin{figure}[htbp]
	\centering
	\includegraphics[scale=0.62]{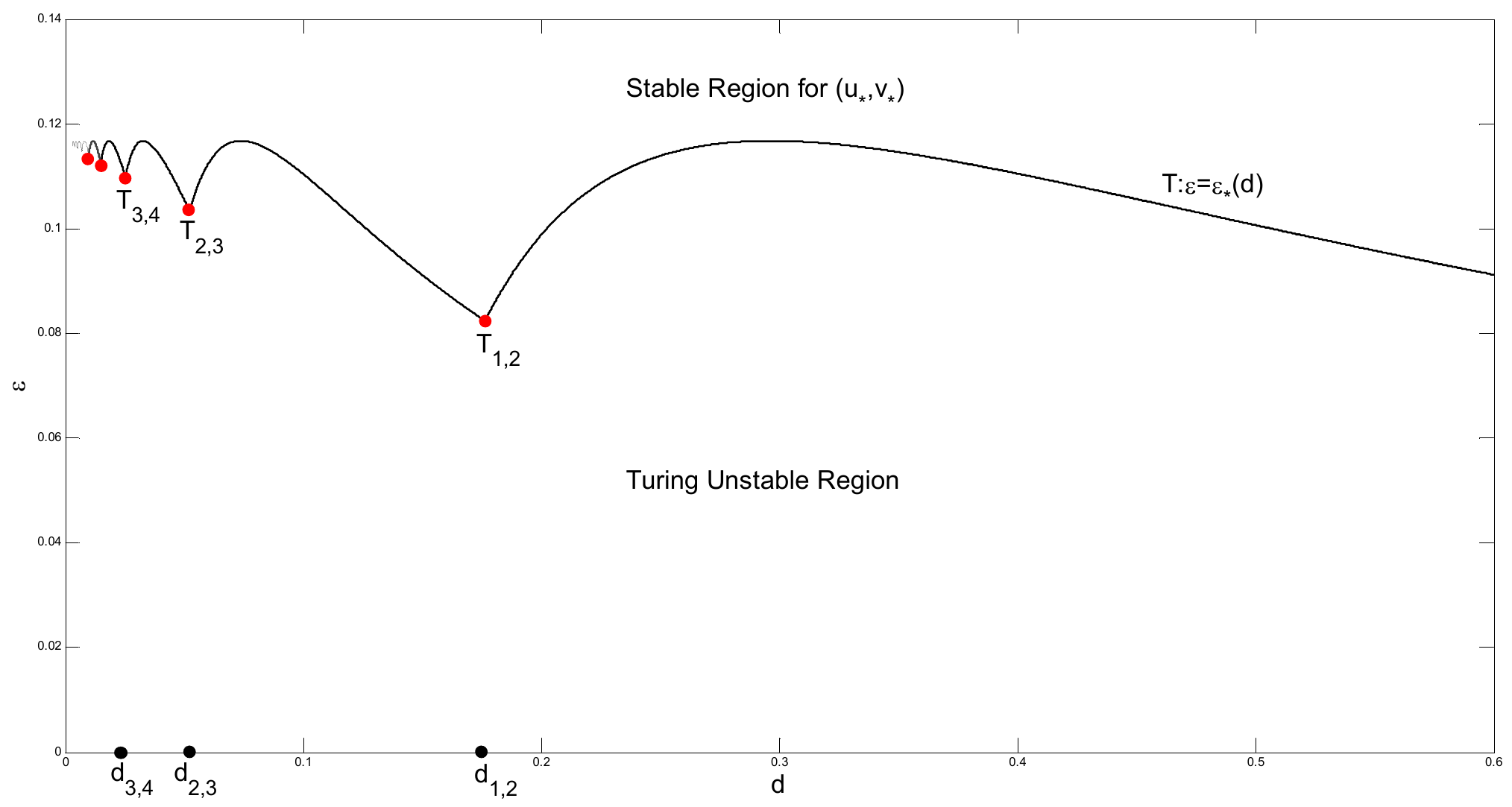}
	\caption{The first Turing bifurcation line $T: ~\varepsilon=\varepsilon_*(d),~d>0$, and Turing-Turing  bifurcation point $T_{k,k+1},k=1,2,3,\cdots$.
	}
	\label{fig32}
\end{figure}

\begin{theorem}\label{T23}
 Suppose that ${\bf (N_0)}$ holds.
\begin{enumerate}
\item[(1)]For any given $k_1\in\mathbb{N}$, we have that
\begin{enumerate}
\item[(a)] When $\varepsilon=\varepsilon_*(k_1,~d)\triangleq \varepsilon_*,~d\in (d_{k_1,k_1+1},d_{k_1-1,k_1}),$
characteristic equation \eqref{eigen} with $k=k_1$ for all $\tau\geq 0$, has a simple real
    eigenvalue $\lambda=\lambda(k_1,\tau,\varepsilon)$  with  $\lambda(k_1,\tau,\varepsilon^*)=0$, $\frac{\mathrm{d }D_{k_1}(\lambda,\tau,\varepsilon)}{\mathrm{d}\varepsilon}|_{\lambda=0,\varepsilon=\varepsilon^*}< 0$,
     and all other roots  of $D_{k}(\lambda,0,\varepsilon_*)$ have negative real parts for $k\in\mathbb{N}_0$.
\item[(b)] At $\varepsilon=\varepsilon_*(k_1,~d)$, system (\ref{3.1b}) 
undergoes $k_1-$mode Turing bifurcation at $(u_*,v_*)$, and the bifurcating  steady state
solutions near $(\varepsilon_*, u_*,v_*)$  can be parameterized as $(\varepsilon(s), u(s),v(s))$, so that $\varepsilon(s)=\varepsilon_*+s$ for
$s\in(-\delta, 0) $ (or $s\in(0, \delta) $) for some small $\delta > 0$, and $(u(s),v(s))=(u_*,v_*)+r\cos (k_1 \pi x)(1,p_{k_1}) $, where $p_{k_1}=\frac{1}{u_*^2}(1-2u_*v_*-d\varepsilon_*\mu_{k_1})$ (see \eqref{pqN}), $r\neq 0$.
\end{enumerate}
\item[(2)] $\varepsilon=\varepsilon_*(d),~d >0$  is the critical curve of Turing  instability.
\begin{enumerate}
\item[(a)] When $\varepsilon>\varepsilon_*(d), ~d>0$, Turing instability will not happen in system (\ref{3.1b}) for all $\tau\geq0$. And equilibrium $(u_*,v_*)$ is asymptotically stable for system (\ref{3.1b}) with $\tau=0$.
    \item[(b)] When $0<\varepsilon<\varepsilon_*(d), ~d>0$, equilibrium $(u_*,v_*)$ is unstable, which is educed by diffusion.
\end{enumerate}
\item[(3)]  On  the critical curve $\varepsilon=\varepsilon_*(d), ~d>0$, $k-$mode Turing bifurcation occurs when $d\in (d_{k,k+1},d_{k-1,k})$,
  and $(k,k+1)-$mode Turing-Turing bifurcation occurs when $d=d_{k,k+1}$, $k\in \mathbb{N}$, which stands for the intersection of two Turing bifurcation curves with different wave numbers $k$ and $k+1$.
\end{enumerate}
\end{theorem}
\begin{proof}
Firstly, ${\bf (N_0)}$ implies that
$$TR_k<TR_0=-1-u_*^2+2u_*v_*<0,~k\in\mathbb{N}.$$
When $d\in (d_{k_1,k_1+1},d_{k_1-1,k_1})$, $\lambda=0$ is a root of $D_{k_1}(\lambda,\tau,\varepsilon_*)$ for all $\tau\geq 0$, and all other roots  of $D_{k_1}(\lambda,\tau,\varepsilon)$ have negative real parts when $\varepsilon=\varepsilon_*$. And
$$DET_{k_1}=0, ~DET_k>0,~k\in\mathbb{N},~k\neq k_1.$$
From a direct calculation, we obtain that
$$
\frac{\mathrm{d }D_{k_1}(\lambda,\tau,\varepsilon)}{\mathrm{d}\lambda}|_{\lambda=0}=p_{k_1}+s_{k_1}-\tau q_{k_1}=-TR_{k_1}+\tau r_{k_1}>0 \  \mathrm{for} \ \tau\geq0,
$$
thus $\lambda=0$ is a simple root.

Next, we show that tranversality condition is valid. Let $\lambda=\lambda(k,\tau,\varepsilon)$ be root of (\ref{eigen}) satisfying $\lambda(k_1,\tau,\varepsilon_*)=0$, then
$$
\frac{\mathrm{d }D_{k_1}(\lambda,\tau,\varepsilon)}{\mathrm{d}\varepsilon}|_{\lambda=0,\varepsilon=\varepsilon_*}=(p_{k_1}+s_{k_1}-\tau q_{k_1})\frac{\mathrm{d}\lambda(k_1,\tau,\varepsilon_*)}{\mathrm{d}\varepsilon}+d^2k_1^4\pi^4+dk_1^2\pi^2 u_*^2=0.
$$
By $\bf (N_0)$ and $DET_{k_1}=0$, we have $p_{k_1}+s_{k_1}-\tau q_{k_1}=u_*^2-2u_*v_*+1+(\varepsilon_*+1)dk_1^2\pi^2+\tau r_{k_1}>0$ when ${\varepsilon=\varepsilon_*}$. Thus
$$ \frac{\mathrm{d}\lambda(k_1,\tau,\varepsilon_*)}{\mathrm{d}\varepsilon}=-\frac{d^2k_1^4\pi^4+dk_1^2\pi^2u_*^2}{p_{k_1}+s_{k_1}-\tau q_{k_1}}<0,\  \mathrm{for} \ \tau\geq0.$$
Combining Lemma \ref{l22}, the theorem is proved. $~~~\Box$
\end{proof}

\begin{corollary}\label{cor:2.4}
 Suppose that ${\bf (N_0)}$ holds. Turing instability (or say Turing bifurcation) does not occur in system (\ref{3.1b}) for all $\tau\geq0$, when one of the following three conditions is established,
\begin{enumerate}
\item[(1)] $\varepsilon>\varepsilon_1$. The critical value $\varepsilon_1$ is a constant, which does not depend on diffusion.
\item[(2)]  $\varepsilon>\varepsilon_B(d),~d>0$. The critical value $\varepsilon_B(d)$ depends on diffusion, but not on mode number $k$.
\item[(3)] $\varepsilon>\varepsilon_*(d),~d>0$. The critical value $\varepsilon_*(d)$ depends on both diffusion and mode number $k$.
\end{enumerate}
Then, equilibrium $(u_*,v_*)$ is asymptotically stable for system (\ref{3.1b}) with $\tau=0$.
\end{corollary}
\begin{remark}
In Corollary \ref{cor:2.4}, terms (1) and (2) both are sufficient conditions, under which Turing instability does not occur, and term (3) is not only sufficient but also necessary. In other words, both $\varepsilon>\varepsilon_1$ and $\varepsilon<\varepsilon_B(d),~d>0$ are necessary for Turing instability, while $\varepsilon<\varepsilon_*(d),~d>0$ is a sufficient and necessary condition. And condition $\varepsilon>\varepsilon_1$, which does not depend on diffusion, is also precisely the (19) in  \cite [Proposition 2]{Hadeler2012Interaction}. The latter two conditions, that is $\varepsilon<\varepsilon_B(d),~d>0$ and $\varepsilon<\varepsilon_*(d),~d>0$, are both  weaker conditions.
\end{remark}
\begin{remark}
From ${\bf (N_1)}$ and ${\bf (N_2)}$, we derive $\varepsilon<1$. Therefore,
Turing instability ( or say Turing bifurcation) does not occur in system (\ref{3.1b}) for $\varepsilon\geq1$.
Notice that $\{(d,\varepsilon):d>\frac{1}{(1-\varepsilon)\pi^2},0<\varepsilon<1\}\subseteq \{(d,\varepsilon):\varepsilon>\varepsilon_B(d),~d>0\}$, which means that we have given a much larger range, where Turing instability does not occur, than \cite{YGLM}. (See Theorem 3.1 (1) of \cite{YGLM} ).
\end{remark}
\begin{remark}
We call the critical curve of Turing  instability $\varepsilon=\varepsilon_*(d),~d >0$
the first Turing bifurcation curve, on which the corresponding characteristic equation without delay has no root with positive real part. It is a piecewise smooth and continuous curve, and piecewise point is exactly Turing-Turing bifurcation point $T_{k.k+1},k\in\mathbb{N}$, see in Figure \ref{fig32}. By \eqref{dkk+} and \eqref{ed}, expression of the first Turing curve explicitly depends on wave number $k$ and diffusion coefficient $d$, so that we can easily find stable spatial pattern with arbitrary wave number.
\end{remark}
\begin{remark} By Theorem 2.1 and Figure \ref{fig32}, we assert that when diffusion ratio $\varepsilon$ is relatively constant, the diffusion coefficient $d$ of activator has great influence on wave number of spatial pattern. Smaller the diffusion coefficient $d$ is, larger the wave number $k$ is.
\end{remark}
\noindent{\bf(2) Delay-induced a Hopf bifurcation}

Suppose that ${\bf (N_0)}$ holds,  and $\varepsilon\geq\varepsilon_*(d),~d>0$. Let $\lambda=\mathrm{i}\omega_k$ with $\omega_k > 0$, be the potential pure
imaginary eigenvalues of characteristic equation (\ref {eigen}) when $\tau=\tau_k$, then
$$
D_k(\mathrm{i}\omega_k, \tau_k,\varepsilon)=r_k-\omega_k^2+q_k \cos(\omega_k \tau_k)+s_k\omega_k \sin(\omega_k \tau_k)+\mathrm{i}[p_k\omega_k+s_k\omega_k \cos(\omega_k\tau_k)-q_k \sin(\omega_k \tau_k)]=0,
$$
$ k\in \mathbb{N}_0$. Thus,
\begin{equation}\label{imagi}
 \begin{array}{rl}
\cos(\omega_k \tau_k)&=\frac{q_k(\omega_k^2-r_k)-p_ks_k\omega_k^2}{s_k^2\omega_k^2+q_k^2} ,\\
\sin(\omega_k \tau_k)&=\frac{s_k\omega_k(\omega_k^2-r_k)+p_kq_k\omega_k}{s_k^2\omega_k^2+q_k^2} .
 \end{array}
\end{equation}
Then, a direct analysis shows
that $\omega_k$ satisfies
\begin{equation}\label{omegak}
\omega_k^4+(p_k^2-s_k^2-2r_k)w_k^2+r_k^2-q_k^2=0.
\end{equation}
 We define
\begin{equation}\label{omegakk}
\omega_k^{\pm}:=\frac{\sqrt{2}}{2}\bigg(s_k^2-p_k^2+2r_k\pm\sqrt{(s_k^2-p_k^2+2r_k)^2
-4(r_k^2-q_k^2)}\bigg)^{1/2}.
\end{equation}

\noindent 1. We firstly consider the sign of $r_k^2-q_k^2=(r_k+q_k)(r_k-q_k)$.

We know that $r_k+q_k=DET_k\geq 0$, and $r_k-q_k=\varepsilon(dk^2\pi^2)^2-(\varepsilon u_*^2-2u_*v_*-1)(dk^2\pi^2)-u_*^2$. Noticing that  ${\bf (N_1)}$ means $\varepsilon u_*^2-2u_*v_*-1<0$, we have
$$
r_k-q_k>r_0-q_0=-u_*^2,\ \mathrm{for}\ \mathrm{any}\  k\in \mathbb{N}.
$$
Defining
\begin{equation}\label{k0}
K  ^0:=K  ^0(\varepsilon)=\frac{1}{\sqrt{2\varepsilon d}\pi}\left[(\varepsilon u_*^2-2u_*v_*-1)+\sqrt{(\varepsilon u_*^2-2u_*v_*-1)^2+4\varepsilon u_*^2}\right]^{\frac{1}{2}},
\end{equation}
we have
$$
r_{K^0}^2-q_{K^0}^2=0,
$$
and
$$
\begin{array}{rlll}
r_k^2-q_k^2<0, & \mathrm{for}& 0\leq &k<K^0,\\
r_k^2-q_k^2>0, & \mathrm{for}& &k>K^0.
\end{array}
$$
Especially, $k_1>K^0(\varepsilon_*)$.

\noindent 2. Then, we consider  $p_k^2-s_k^2-2r_k$, which is regarded as a function of $k$.

By $p_k^2-s_k^2-2r_k=({\varepsilon}^2+1)d^2k^4\pi^4+2\varepsilon dk^2\pi^2-(u_*^2-2u_*v_*)^2+1$, define
\begin{equation}\label{k00}
K  ^{+}=\frac{1}{\pi\sqrt{(\varepsilon^2+1) d}}\left(-\varepsilon +\sqrt{(u_*^2-2u_*v_*)^2(\varepsilon^2+1)- 1}\right)^{1/2},~ \mathrm{for} ~ (u_*^2-2u_*v_*)^2(\varepsilon^2+1)\geq 1,
\end{equation}
then $p_{K  ^{+}}^2-s_{K  ^{+}}^2-2r_{K  ^{+}}=0$, and $K^+$ is the only one, which may be positive root. We have following results.
\begin{theorem}\label{thm:2.9}
Suppose that ${\bf (N_0)}$ holds,  and $\varepsilon\geq\varepsilon_*(d),~d>0$. Then for $k\in \mathbb{N}_0$,
\begin{enumerate}
\item when $0\leq k<K^0$, equation (\ref{omegak}) has a positive root $\omega_k^+$.
\item
\begin{enumerate}
\item  if $(u_*^2-2u_*v_*)^2(\varepsilon^2+1)<1$, equation (\ref{omegak}) has no positive root when $k\geq K^0$.
\item  if $(u_*^2-2u_*v_*)^2(\varepsilon^2+1)\geq 1$,
\begin{enumerate}
\item  when $K^0\geq K^+$, equation (\ref{omegak}) has no positive root for $k\geq K^0$.
\item when $K^0<K^+$, there is a $K_*\in(K^0,K^+)$, such that equation (\ref{omegak}) has two positive roots $\omega_k^{\pm}$, a positive root $\omega_k^{+}$, and has no positive root for $k\in (K^0,K_*),~k= K^0~\mathrm{or}~K_*$, $k>K_*$, respectively.
\end{enumerate}
\end{enumerate}
\end{enumerate}
\end{theorem}

\begin{proof}
Conclusion 1 is obvious.
\end{proof}

Under the conditions of 2(a), we have $s_k^2-p_k^2+2r_k<0$ and $r_k^2-q_k^2>0$ for $k>K^0$, so equation (\ref{omegak}) has no positive root for $k>K^0$.

If conditions of 2(b)-i are satisfied,  by $k\geq K^0$ we have $s_k^2-p_k^2+2r_k<0$ and $r_k^2-q_k^2>0$, which means that $\omega_k^{-}$ and $\omega_k^{+}$ are not positive roots.

When conditions of 2(b)-ii hold, letting
$$
\Delta (k):=(s_k^2-p_k^2+2r_k)^2
-4(r_k^2-q_k^2),$$
we have $
\Delta (K^0)=(s_{K^0}^2-p_{K^0}^2+2r_{K^0})^2
>0$ and $\Delta (K^+)=
-4(r_{K^+}^2-q_{K^+}^2)<0$. So there exist $K_*\in(K^0,K^+)$ such that
$\Delta (K_*)=0 $, and $\Delta (k)>0,~r_k^2-q_k^2>0$ for $k\in (K^0,K_*)$ which means $\omega_k^{-}$ and $\omega_k^{+}$ are positive roots. By $(s_{K^0}^2-p_{K^0}^2+2r_{K^0})^2
>0$ and $r_{K^0}^2-q_{K^0}^2=0$,  we have $\omega_{K^0}^{-}=0$, so there is only one positive root $\omega_{K^0}^{+}$ at $K^0$. By $\Delta (K_*)=0$ and $(s_{K_*}^2-p_{K_*}^2+2r_{K_*})^2>0$, we have $\omega_{K_*}^+=\omega_{K_*}^- >0$, so there is also only one positive root at $K_*$. $\Box$

Based on Theorem \ref{thm:2.9}, we have following corollary.

\begin{corollary}
Suppose that ${\bf (N_0)}$ holds,  and $\varepsilon\geq\varepsilon_*(d),~d>0$, we define that
\begin{equation}\label{k*}
K^*:=\left\{\begin{array}{ll} K^0,
&  \mathrm{for} ~(u_*^2-2u_*v_*)^2(\varepsilon^2+1)<1\\
K^0, & \mathrm{for} ~(u_*^2-2u_*v_*)^2(\varepsilon^2+1)\geq 1~ \mathrm{and} ~K^0\geq K^+,\\
K_*, & \mathrm{for} ~(u_*^2-2u_*v_*)^2(\varepsilon^2+1)\geq 1~ \mathrm{and} ~K^0< K^+,
\end{array}\right.
\end{equation}
then equation (\ref{omegak}) has at least a positive root $\omega_k^+$ for $0\leq k<K^*$ and $k\in \mathbb{N}_0$.
\end{corollary}

Thus we define $\tau_k\in (0,2\pi]$ which is a root of (\ref{imagi}), and
\begin{equation}\label{tau}
\tau_k^{(j)}=\tau_k+\frac{2\pi j}{\omega_k^+},~j,k=0,1,2,\ldots, 0\leq k<K^*.
\end{equation}
where $\pm\mathrm{i}\omega_k^+$ are pure
imaginary eigenvalues of characteristic equation (\ref {eigen}) when $\tau=\tau_k^{(j)},~j=0,1,2,\ldots, ~0\leq k<K^*$.  Then following theorem on tranversality condition on standard Hopf bifurcation theorem holds:

\begin{theorem}\label{thm:2.11}
 Suppose that ${\bf (N_0)}$ holds,  and $\varepsilon\geq\varepsilon_*(d),~d>0$.  Let $\tau_k^{(j)}$ be defined as in (\ref{tau}), and $\lambda(\tau)=\alpha(\tau)+\mathrm{i}\omega(\tau)$ be a root of $D_k(\lambda,\tau)=0$  in (\ref {eigen}) near
$\tau=\tau_k^{(j)}$ satisfying $\alpha(\tau_k^{(j)})=0, ~\omega(\tau_k^{(j)})=\omega_k^+,~k\in \mathbb{N}_0,~0\leq k<K^*$ . Then
$$\frac{\mathrm{d Re}\lambda(\tau)}{\mathrm{d}\tau}|_{\tau_k^{(j)}}>0.$$
\end{theorem}
See \cite {YGLM} for the proof of Theorem 2.2.

We further assume that
\begin{enumerate}
\item[${\bf (N_3)}$] There is an integer $k_2 \in [0,K^*)$ satisfying that for any integer $k \in [0,K^*)$, we have
  $$
  \begin{array}{rll}
  \tau_{k} =& \operatorname*{min}\limits_{0\leq s<K^*,s\in \mathbb{N}_0}\tau_s  ,~~~&k=k_2,\\
  \tau_{k}> & \operatorname*{min}\limits_{0\leq s<K^*,s\in \mathbb{N}_0}\tau_s ,~~~&k\neq k_2.
\end{array}
$$
\end{enumerate}

So far, we summarize our results on the stability of $(u_*,v_*)$ and Hopf bifurcation of system (\ref {3.1b}) in following theorem.

\begin{theorem}\label{thm:2.12}
Suppose that ${\bf (N_0)}$  and ${\bf (N_3)}$ hold,  and $\varepsilon>\varepsilon_*(d),~d>0$. Then
\begin{enumerate}
\item At $\tau=\tau_k^{(j)}$ with $0\leq k<K^*$ and $j, k\in \mathbb{N}_0$, system (\ref {3.1b}) undergoes $k-$mode Hopf bifurcation near $(u_*,v_*)$, and the bifurcating  periodic
solutions near $(\tau_k^{(j)}, u_*,v_*)$  can be parameterized as $(\tau(s), u(s),v(s))$ so that $\tau(s)=\tau_k^{(j)}+s$ for
$s\in(-\delta, 0) $ (or $s\in(0, \delta) $) for some small $\delta > 0$, and $(u(s),v(s))=(u_*,v_*)+[r_1(1,p_{k}^0)e^{\texttt{i}\tau_{k}\omega_{k}^+ \theta}+r_2(1,\overline{p_{k}^0})e^{-\texttt{i}\tau_{k}\omega_{k}^+ \theta}]\cos (k \pi x),~~-1\leq \theta \leq 0 $, where $p_{k}^0=\frac{1}{u_*^2}(1-2u_*v_*e^{-\mathrm{i}\tau_{k}\omega_{k}^+}
-d\varepsilon_*\mu_{k}+\mathrm{i}\omega_{k}^+)e^{\mathrm{i}\tau_{k}\omega_{k}^+}$ (see \eqref{pqN}), $r_1~\mathrm{or}~r_2\neq 0$.
\item When $\tau=\tau_{k_2}$,  $D_{k_2}(\lambda,\tau,\varepsilon)$ has a pair of pure imaginary roots, with all other roots of $D_{k_2}(\lambda,\tau,\varepsilon)$ and all roots  of $D_{k}(\lambda,\tau,\varepsilon),k\neq k_2$  having negative real parts. And equilibrium $(u_*,v_*)$ is locally asymptotically stable in system (\ref{3.1b}) with $\tau\in [0,\tau_{k_2})$.
\end{enumerate}
\end{theorem}
\begin{remark} When $\omega_k^-$ is also a positive root, we can analogously define $\tau_k^-\in (0,2\pi]$ which is a root of (\ref{imagi}), and define
\begin{equation}\label{tau-}
\tau_k^{(j-)}=\tau_k^-+\frac{2\pi j}{\omega_k^-},~j,k=0,1,2,\ldots, K^ 0\leq k<K_*,
\end{equation}
where $\pm\mathrm{i}\omega_k^-$ are corresponding pure
imaginary eigenvalues of characteristic equation (\ref {eigen}). The tranversality condition can be accordingly described as:
$$\frac{\mathrm{d Re}\lambda(\tau)}{\mathrm{d}\tau}|_{\tau_k^{(j-)}}<0.$$
Thus, when all roots  of $D_{k}(\lambda,\tau,\varepsilon)$ with $\tau=0$ have negative real parts for any $k\in \mathrm{N}_0$, we assert that
$$
\operatorname*{min}\limits_{0\leq s<K^*,s\in \mathbb{N}_0}\tau_s< \operatorname*{min}\limits_{K^0\leq s<K_*,s\in \mathbb{N}_0}\tau_s^-.
$$
\end{remark}
\begin{remark}
 By theorem 1 and 4, Hopf bifurcation doesn't occurs in system (\ref{3.1b}) with diffusion and without delay. Hence, Hopf bifurcation is induced by delay, that is, delay-driven oscillation occurs.
\end{remark}
\noindent {\bf(3) Diffusion and Delay-induced a Turing-Hopf bifurcation}

From above discussions, it follows that, for any $\tau > 0$ and $d=0$, the FDEs does not undergo any
Hopf-steady state bifurcation, and  for any $d > 0$ and $\tau =0$, the PDEs system does not undergo any
Hopf-steady state bifurcation. However, in this section, we shall show the joint effects of diffusion and delay, which induce otherwise non-existing Hopf-steady state bifurcation, or rather, it is a Turing-Hopf bifurcation in PFDEs.
By {\bf Theorem \ref{T23}, \ref{thm:2.9}, \ref{thm:2.11}, \ref{thm:2.12}} and \cite [Definition 2.1] {JAS}, we summarize our results on the stability of $(u_*,v_*)$ and Turing-Hopf bifurcation of system (\ref {3.1b}) in following theorem.

\begin{theorem}\label{thm:2.15}
Suppose that ${\bf (N_0)}$ and ${\bf (N_3)}$ hold. Then
\begin{enumerate}
\item System (\ref {3.1b}) undergoes $(k_1,k_2)$-mode Turing-Hopf bifurcation near $(u_*,v_*)$  at $\varepsilon=\varepsilon_*(k_1,d)\triangleq\varepsilon_*$, $\tau=\tau_{k_2}$ for $d\in (d_{k_1,k_1+1},d_{k_1-1,k_1})$.
\item The equilibrium $(u_*,v_*)$ is asymptotically stable in system (\ref{3.1b}) with $\tau\in [0,\tau_{k_2})$ for $\varepsilon>\varepsilon_*$, and unstable for $0<\varepsilon<\varepsilon_*$.
\end{enumerate}
\end{theorem}

\section{Third-order normal form of Turing-Hopf bifurcation}
In this section, we are interested in determining the third-order normal forms with
original perturbation parameters for system
(\ref{3.1b})  with $(k_1,k_2)$-mode Turing-Hopf bifurcation when $(\tau, \varepsilon)$ near $(\tau_{k_2}, \varepsilon_*)$, according to the formula in \cite{JAS}. Rewrite $\tau=\tau_{k_2}+\alpha_1,
\varepsilon=\varepsilon_*+\alpha_2$. Then for
$(u_*,v_*)$ and $\alpha_1=0$, $\alpha_2=0$, system (\ref{3.1b}) exhibits $(k_1,k_2)$-mode Turing-Hopf bifurcation. 
We normalize the
delay $\tau$ in system (\ref{3.1b}) by time-scaling $t\rightarrow
t/\tau$, and translate $(u_*,v_*)$ into origin. Then, system
(\ref{3.1b}) is transformed into
\begin{equation}\label{rdu2}
  \begin{array}{rl}
\frac{\partial u}{dt}=&\tau\big(\varepsilon du_{xx}+a-(u(t,x)+u_*)+(u(t-1,x)+u_*)^2(v(t-1,x)+v_*)\big), \\
\frac{\partial v}{dt}=&\tau\big(d v_{xx}+b-(u(t-1,x)+u_*)^2(v(t-1,x)+v_*)\big). \\
  \end{array}
\end{equation}
The corresponding characteristic equations are
\begin{equation}
\lambda^2+\tau p_k(\varepsilon)\lambda+\tau^2r_k(\varepsilon)+(\tau s_k\lambda+\tau^2 q_k(\varepsilon))e^{-\lambda}=0,\;k\in\mathbb{N}_0,
\end{equation}

Define $U(t)=(u(t),v(t))$, and introduce two bifurcation parameters
$\alpha=(\alpha_1,\alpha_2)$ by setting
\begin{equation}
\tau=\tau_{k_2}+\alpha_1,\;\; \varepsilon=\varepsilon_*+\alpha_2.
\end{equation}
Then, 
system (\ref{rdu2}) can be written in following form in $\mathcal{C}=C([-1,0];X)$:
\begin{equation}\label{lin}
\frac{d}{dt}U(t)=L_0(U_t)+D_0\Delta U(t)+\frac{1}{2}L_1(\alpha)U_t+\frac{1}{2}D_1(\alpha)\Delta U(t)+\frac{1}{2!}Q(U_t,U_t)+\frac{1}{3!}C(U_t,U_t,U_t)+\cdots,
\end{equation}
 with $X=\left\{(u,v):u,v\in W^{2,2}(\Omega): u'(0)=u'(1)=v'(0)=v'(1)=0\; \right\},$ $\Omega=(0,1)$,
and
$$
D_0=d\tau_{k_2} \left(
\begin{array}{cc}
\varepsilon_*&0\\
0&1
\end{array}\right),
$$
$$
D_1(\alpha)=2d\left(
\begin{array}{cc}
\alpha_1\varepsilon_*+\alpha_2\tau_{k_2}&0\\
0&\alpha_1
\end{array}\right),
$$
$$
L_0\varphi=\tau_{k_2}\left(
\begin{array}{c}
-\varphi^{(1)}(0)+2u_*v_*\varphi^{(1)}(-1)+u_*^2\varphi^{(2)}(-1) \\
-(2u_*v_*\varphi^{(1)}(-1)+u_*^2\varphi^{(2)}(-1))
\end{array}\right),
$$
$$
L_1(\alpha)\varphi=2\alpha_1\left(
\begin{array}{c}
-\varphi^{(1)}(0)+2u_*v_*\varphi^{(1)}(-1)+u_*^2\varphi^{(2)}(-1) \\
-(2u_*v_*\varphi^{(1)}(-1)+u_*^2\varphi^{(2)}(-1))
\end{array}\right),
$$
$$
Q(\varphi,\varphi)=2\tau_{k_2}\varphi^{(1)}(-1)[v_*\varphi^{(1)}(-1)+2u_*\varphi^{(2)}(-1)]\left(
\begin{array}{c}
1\\
-1
\end{array}\right),
$$
$$
C(\varphi,\varphi,\varphi)=3!\tau_{k_2}{\varphi^{(1)}}^2(-1)\varphi^{(2)}(-1)\left(
\begin{array}{c}
1\\
-1
\end{array}\right),~~
\varphi=\left(
\begin{array}{c}
\varphi^{(1)}\\
\varphi^{(2)}
\end{array}\right).
$$
Thus,
$$
Q_{XY}=2\tau_{k_2}\{v_*x_1(-1)y_1(-1)+u_*(x_1(-1)y_2(-1)+x_2(-1)y_1(-1))\}\left(
\begin{array}{c}
1\\
-1
\end{array}\right),
$$
$$
C_{XYZ}=2\tau_{k_2}[x_1(-1)y_1(-1)z_2(-1)+x_1(-1)y_2(-1)z_1(-1)+x_2(-1)y_1(-1)z_1(-1)]\left(
\begin{array}{c}
1\\
-1
\end{array}\right),
$$
where $
X=\left(
\begin{array}{c}
x_1\\
x_2
\end{array}\right), ~Y=\left(
\begin{array}{c}
y_1\\
y_2
\end{array}\right),~Z=\left(
\begin{array}{c}
z_1\\
z_2
\end{array}\right).$
The corresponding characteristic equations with $\tau=\tau_{k_2},\varepsilon=\varepsilon_*$ are represented as
\begin{equation}\label{eigen-11}
D_k(\lambda):=\lambda^2+\tau_{k_2}p_k(\varepsilon_*)\lambda+\tau_{k_2}^2r_k(\varepsilon_*)+(\tau_{k_2}s_k\lambda+\tau_{k_2}^2q_k(\varepsilon_*))e^{-\lambda}=0,\;k\in\mathbb{N}_0.
\end{equation}
By \cite{JAS}, we know that the normal forms restrict on center manifold up to the third order are
\begin{equation}\label{third}
	\begin{array}{rlc}
	\dot{z}_1=&a_{1}(\alpha)z_1+a_{11}z_1^2+a_{23}z_2\bar{z}_2++a_{111}z_1^3+ a_{123}z_1z_2\bar{z}_2+h.o.t.,\\
	 \dot{z}_2=&\mathrm{i}\omega_0z_2+b_{2}(\alpha)z_2+b_{12}z_1z_2+b_{112}z_1^2z_2+b_{223}z_2^2\bar{z}_2
	+ h.o.t. .\\
	 \dot{\bar{z}}_2=&-\mathrm{i}\omega_0\bar{z}_2+\overline{b_{2}(\alpha)}\bar{z}_2+\overline{b_{12}}z_1\bar{z}_2+\overline{b_{112}}z_1^2\bar{z}_2+\overline{b_{223}}z_2\bar{z}_2^2+
	h.o.t. ,
	\end{array}
	\end{equation}
Here, we only consider the case $k_1\neq 0,~k_2=0$ in \cite{JAS}, which is one of the most interesting and practical situations.
\begin{lemma} (see \cite{JAS}) \label{L1}	
For $k_2=0,~k_1\neq 0$ and Neumann boundary condition on spatial domain $\Omega =(0,l\pi),~l>0$, the parameters $a_{1}(\alpha),~b_{2}(\alpha),~a_{11},~a_{23},~a_{111},~a_{123},~b_{12},~b_{112},~b_{223}$ in (\ref{third}) are
	\begin{equation}\label{eq03-3++}
	\begin{array}{rlc}
    a_{1}(\alpha)=&\frac{1}{2}\psi_1 (0)(L_1(\alpha)\phi_1-\mu_{k_1}D_1(\alpha)\phi_1(0)),\\

	b_{2}(\alpha)=&\frac{1}{2}\psi_2 (0)(L_1(\alpha)\phi_2-\mu_{k_2}D_1(\alpha)\phi_2(0)),\\

	a_{11}=&a_{23}=b_{12}=0,\\
	
	a_{111}=&\frac{1}{4}\psi_1 (0)C_{\phi_1\phi_1\phi_1}+\frac{1}{\omega_0}\psi_1 (0)\mathrm{Re}(\mathrm{i}Q_{\phi_1\phi_2}\psi_2 (0))Q_{\phi_1\phi_1}+\psi_1 (0) Q_{\phi_1} ( h_{200}^{0}+\frac{1}{\sqrt{2}}h_{200}^{2k_1}),\\
	
	a_{123}=&\psi_1 (0)C_{\phi_1\phi_2\bar{\phi}_2}+\frac{2}{\omega_0}\psi_1 (0)\mathrm{Re}(\mathrm{i}Q_{\phi_1\phi_2}\psi_2 (0))Q_{\phi_2\bar{\phi}_2}+\\
	&\psi_1 (0) [Q_{\phi_1} ( h_{011}^{0}+\frac{1}{\sqrt{2}}h_{011}^{2k_1})
	+Q_{\phi_2}  h_{101}^{k_1}
	+Q_{\bar{\phi}_2}  h_{110}^{k_1}],\\
	
	b_{112}=&\frac{1}{2}\psi_2 (0)C_{\phi_1\phi_1\phi_2}+\frac{1}{2i\omega_0}\psi_2 (0)\{2Q_{\phi_1\phi_1}\psi_1 (0) Q_{\phi_1\phi_2} +[-Q_{\phi_2\phi_2}\psi_2 (0)
	
	+Q_{\phi_2\bar{\phi}_2}\bar{\psi}_2 (0)]Q_{\phi_1\phi_1} \}\\&+\psi_2 (0) (Q_{\phi_1}  h_{110}^{k_1}+Q_{\phi_2}  h_{200}^{0})\\
	
	b_{223}=&\frac{1}{2}\psi_2 (0)C_{\phi_2\phi_2\bar{\phi}_2}+\frac{1}{4i\omega_0}\psi_2 (0)\{
	\frac{2}{3}Q_{\bar{\phi}_2\bar{\phi}_2}\bar{\psi}_2 (0)Q_{\phi_2\phi_2}+[-2Q_{\phi_2\phi_2}\psi_2 (0)\\&
	
	+4Q_{\phi_2\bar{\phi}_2}\bar{\psi}_2 (0)]Q_{\phi_2\bar{\phi}_2}\}+\psi_2 (0) (Q_{\phi_2} h_{011}^0+Q_{\bar{\phi}_2} h_{020}^0).
	\end{array}
	\end{equation}
	where
\begin{equation}\label{eq03-3}
	\begin{array}{rlc}
	h_{200}^0(\theta)=&-\frac{1}{2}[\int_{-r}^0\mathrm{d}\eta _0(\theta)]^{-1}Q_{\phi_1\phi_1}+\frac{1}{2i\omega_0}(\phi_2(\theta)\psi_2(0)-\bar{\phi}_2(\theta)\bar{\psi}_2(0))]Q_{\phi_1\phi_1},\\
	
	h_{200}^{2k_1}(\theta)\equiv &-\frac{1}{2\sqrt{2}}[\int_{-r}^0\mathrm{d}\eta _{2k_1}(\theta)]^{-1}Q_{\phi_1\phi_1},\\
	
	h_{011}^0(\theta)=&-[\int_{-r}^0\mathrm{d}\eta _0(\theta)]^{-1}Q_{\phi_2\bar{\phi}_2}+\frac{1}{i\omega_0}(\phi_2(\theta)\psi_2(0)-\bar{\phi}_2(\theta)\bar{\psi}_2(0))]Q_{\phi_2\bar{\phi}_2},\\
	
	h_{011}^{2k_1}(\theta)=&0,\\
	
	 h_{020}^0(\theta)=&\frac{1}{2}[2i\omega_0I-\int_{-r}^0e^{2i\omega_0\theta}\mathrm{d}\eta _{0}(\theta)]^{-1}Q_{\phi_2\phi_2}e^{2i\omega_0\theta}-\frac{1}{2i\omega_0}[\phi_2(\theta)\psi_2(0)+\frac{1}{3}\bar{\phi}_2(\theta)\bar{\psi}_2(0)]Q_{\phi_2\phi_2},\\
	
	 h_{110}^{k_1}(\theta)=&[i\omega_0I-\int_{-r}^0e^{i\omega_0\theta}\mathrm{d}\eta _{k_1}(\theta)]^{-1}Q_{\phi_1\phi_2}e^{i\omega_0\theta}-\frac{1}{i\omega_0}\phi_1(0)\psi_1(0)Q_{\phi_1\phi_2},\\
	
	 h_{002}^0(\theta)=&\overline{h_{020}^0(\theta)},~~~h_{101}^{k_1}(\theta)=\overline{h_{110}^{k_1}(\theta)}.
	\end{array}
	\end{equation}
	 $\theta\in[-r,0]$, $\phi_1,\phi_2,\psi_1(0),\psi_2(0)$  see \cite[(2.8)] {JAS}, and $\eta_k\in BV([-r,0],\mathbb{C}^m)$ is denoted by \cite [(2.6)] {JAS}, that is
\begin{equation}\label{eq303}
-\mu_kD_0 \psi(0)+L_0\psi = \int_{-r}^0 \mathrm{d}\eta_k(\theta)\psi(\theta),\;\; ~\psi\in C\triangleq C([-r,0],\mathbb {C}^m),
\end{equation}
$k\in \mathbb{N}_0$.
\end{lemma}
Now, we are going to calculate coefficients in the third-order normal form (\ref{third}) by explicit formulas (\ref{eq03-3++}) and (\ref{eq03-3}).

\noindent{\bf(3-1)} To get expressions of $\phi_1,\phi_2,\psi_1(0),\psi_2(0)$, and further to get the expression of $a_1(\alpha),b_2(\alpha)$  in (\ref{third}) and $Q{\phi_i\phi_j}$, $C{\phi_i\phi_j\phi_l},\ i,j,l=1,2$.

By \cite[(2.8)] {JAS}, and noticing $\omega_0=\tau_{k_2}\omega_{k_2}^+$, we have

$$
\phi_1(\theta ) =\left(
\begin{array}{c}
1 \\
p_1^0
\end{array}\right),~~
\phi_2(\theta ) =\left(
\begin{array}{c}
1 \\
p_2^0
\end{array}\right)e^{\texttt{i}\tau_{k_2}\omega_{k_2}^+ \theta},~~-1\leq \theta \leq 0
$$
and
$$
\psi_1(0 ) =\frac{1}{N_1}\left(
1 ,
q_1^0
\right),~~
\psi_2(0) =\frac{1}{N_2}\left(
1 ,
q_2^0
\right),
$$
where
\begin{equation}\label{pqN}
\begin{array}{rl}
q_1^0=&-\frac{1}{2u_*v_*}(1-2u_*v_*-d\varepsilon_*\mu_{k_1}),\\
p_1^0=&\frac{1}{u_*^2}(1-2u_*v_*-d\varepsilon_*\mu_{k_1}),\\
q_2^0=&-\frac{1}{2u_*v_*}(1-2u_*v_*e^{-\mathrm{i}\tau_{k_2}\omega_{k_2}^+}-d\varepsilon_*\mu_{k_2}+\mathrm{i}\omega_{k_2}^+)e^{\mathrm{i}\tau_{k_2}\omega_{k_2}^+},\\
p_2^0=&\frac{1}{u_*^2}(1-2u_*v_*e^{-\mathrm{i}\tau_{k_2}\omega_{k_2}^+}-d\varepsilon_*\mu_{k_2}+\mathrm{i}\omega_{k_2}^+)e^{\mathrm{i}\tau_{k_2}\omega_{k_2}^+},\\
N_1=&1+p_1^0q_1^0+\tau_{k_2}
u_*(2v_*+u_*p_1^0)(1-q_1^0),\\
N_2=&1+p_2^0q_2^0+\tau_{k_2}
u_*(2v_*+u_*p_2^0)(1-q_2^0)e^{-\mathrm{i}\tau_{k_2}\omega_{k_2}^+}.
\end{array}
\end{equation}
So, we have
\begin{equation}\label{a1}
\begin{array}{rlc}
a_{1}(\alpha)
=&\frac{1}{N_1}[-1+u_*(2v_*+u_*p_1^0)(1-q_1^0)-dk_1^2\pi^2(\varepsilon_*+p_1^0q_1^0)]\alpha_1-
\frac{1}{N_1}dk_1^2\pi^2\tau_{k_2}\alpha_2,\\

b_{2}(\alpha)
=&\frac{1}{N_2}[-1+u_*(2v_*+u_*p_2^0)(1-q_2^0)e^{-\texttt{i}\tau_{k_2}\omega_{k_2}^+}-dk_2^2\pi^2(\varepsilon_*+p_2^0q_2^0)]\alpha_1-\frac{1}{N_2}dk_2^2\pi^2\tau_{k_2}\alpha_2.
\end{array}
\end{equation}
\begin{equation}\label{Q}
\begin{array}{rl}
Q_{\phi_1\phi_1}=& 2\tau_{k_2}(v_*+2u_*p_1^0)(
1,\ -1)^{T},\\

Q_{\phi_1\phi_2}=& 2\tau_{k_2}e^{-\texttt{i}\tau_{k_2}\omega_{k_2}^+ }[v_*+u_*(p_1^0+p_2^0)](
1,\ -1)^{T},\\

Q_{\phi_1\bar{\phi}_2}=& 2\tau_{k_2}e^{\texttt{i}\tau_{k_2}\omega_{k_2}^+ }[v_*+u_*(p_1^0+\overline{p_2^0})](
1,\ -1)^{T},\\

Q_{\phi_2\phi_2}=& 2\tau_{k_2}e^{-2\texttt{i}\tau_{k_2}\omega_{k_2}^+ }(v_*+2u_*p_2^0)(
1,\ -1)^{T},\\

Q_{\phi_2\bar{\phi}_2}=& 2\tau_{k_2}(v_*+u_*(p_2^0+\overline{p_2^0}))(
1,\ -1)^{T},\\

Q_{\bar{\phi}_2\bar{\phi}_2}=& 2\tau_{k_2}e^{2\texttt{i}\tau_{k_2}\omega_{k_2}^+ }(v_*+2u_*\overline{p_2^0})(
1,\ -1)^{T}
\end{array}
\end{equation}
and
\begin{equation}\label{C}
\begin{array}{rl}
C_{\phi_1\phi_1\phi_1}=& 6\tau_{k_2}p_1^0(
1,\ -1)^{T},\\

C_{\phi_1\phi_2\bar{\phi}_2}=& 2\tau_{k_2}(p_1^0+p_2^0+\overline{p_2^0})](
1,\ -1)^{T},\\

C_{\phi_1\phi_1\phi_2}=& 2\tau_{k_2}e^{-\texttt{i}\tau_{k_2}\omega_{k_2}^+ }(2p_1^0+p_2^0)(
1,\ -1)^{T},\\

C_{\phi_2\phi_2\bar{\phi}_2}=& 2\tau_{k_2} (2p_2^0+\overline{p_2^0}) e^{-\texttt{i}\tau_{k_2}\omega_{k_2}^+ }(
1,\ -1)^{T}.
\end{array}
\end{equation}

\noindent{\bf(3-2)} To calculate expressions of $h_{200}^0,~h_{200}^{2k_1},~h_{011}^0,~h_{011}^{2k_1},~h_{110}^{k_1},~h_{101}^{k_1},~h_{020}^{0}$ by(\ref{eq03-3}) and (\ref{eq303}).

\begin{equation}\label{eq03-3-h}
\begin{array}{rlc}
h_{200}^0(\theta)
                 =&(v_*-2u_*p_1^0)\left[\frac{1}{u_*^2}\left(
\begin{array}{c}
0 \\
-1
\end{array}\right)+\frac{2}{\omega_{k_2}^+}\left(
\begin{array}{c}
\mathrm{Re}(\frac{1-q_2^0}{\texttt{i}N_2}e^{\texttt{i}\tau_{k_2}\omega_{k_2}^+ \theta}) \\
\mathrm{Re}(\frac{(1-q_2^0)p_2^0}{\texttt{i}N_2}e^{\texttt{i}\tau_{k_2}\omega_{k_2}^+ \theta})
\end{array}\right)\right],\\

h_{200}^{2k_1}(\theta)
                            =&-\frac{1}{\sqrt{2}}(v_*+2u_*p_1^0)\frac{1}{q_{2k_1}(\varepsilon_*)+r_{2k_1}(\varepsilon_*)}\left(
\begin{array}{c}
-(2k_1)^2\pi^2d \\
1+(2k_1)^2\pi^2d\varepsilon_*
\end{array}\right)\\

h_{011}^0(\theta)
                 =&2[v_*+u_*(p_2^0+\overline{p_2^0})]\left[\frac{1}{u_*^2}\left(
\begin{array}{c}
0 \\
-1
\end{array}\right)+\frac{2}{\omega_{k_2}^+}\left(
\begin{array}{c}
\mathrm{Re}(\frac{1-q_2^0}{\texttt{i}N_2}e^{\texttt{i}\tau_{k_2}\omega_{k_2}^+ \theta}) \\
\mathrm{Re}(\frac{(1-q_2^0)p_2^0}{\texttt{i}N_2}e^{\texttt{i}\tau_{k_2}\omega_{k_2}^+ \theta})
\end{array}\right)\right],\\
h_{011}^{2k_1}(\theta)=&0,\\

h_{020}^0(\theta)
                 =&(v_*+2u_*p_1^0)\frac{e^{\texttt{i}\tau_{k_2}\omega_{k_2}^+ (\theta-2)}}{D_0(2\texttt{i}\omega_{k_2}^+,\tau_{k_2},\varepsilon_*)}\left(
\begin{array}{c}
2\texttt{i}\omega_{k_2}^+ \\
-2\texttt{i}\omega_{k_2}^+ -1
\end{array}\right)\\
                   &-\frac{(v_*+2u_*p_1^0)}{\texttt{i}\omega_{k_2}^+}(e^{\texttt{i}\tau_{k_2}\omega_{k_2}^+(\theta-2)}\frac{(1-q_2^0)}{N_2}\left(
\begin{array}{c}
1 \\
p_2^0
\end{array}\right)+\frac{1}{3} e^{-\texttt{i}\tau_{k_2}\omega_{k_2}^+ (\theta+2)}\frac{(1-\overline{q_2^0})}{\overline{N_2}}\left(
\begin{array}{c}
1 \\
\overline{p_2^0}
\end{array}\right))\\

h_{002}^0(\theta)=&\overline{h_{020}^0(\theta)},\\

h_{110}^{k_1}(\theta)
                     =&2[v_*+u_*(p_1^0+p_2^0)]e^{\texttt{i}\tau_{k_2}\omega_{k_2}^+ (\theta-1)}\frac{1}{
                     D_{k_1}(\texttt{i}\omega_{k_2}^+,\tau_{k_2},\varepsilon_*)}\left(
\begin{array}{c}
\texttt{i}\omega_{k_2}^++d k_1^2\pi^2 \\
-\texttt{i}\omega_{k_2}^+-\varepsilon_*d k_1^2\pi^2-1
\end{array}\right)\\
                   &-2[v_*+u_*(p_1^0+p_2^0)]e^{-\texttt{i}\tau_{k_2}\omega_{k_2}^+}\frac{(1-q_1^0)}{\texttt{i}\omega_{k_2}^+N_1}\left(
\begin{array}{c}
1 \\
p_1^0
\end{array}\right),\\

h_{101}^{k_1}(\theta)=&\overline{h_{110}^{k_1}(\theta)}.
\end{array}
\end{equation}
Plugging \eqref{Q},\eqref{C} and \eqref{eq03-3-h} into \eqref{eq03-3++}, expressions of $a_{111}, a_{123}, b_{123}, b_{223}$ are obtained.  Further, we have following result by \cite{JAS}.

 \begin{theorem}\label{amplitude}
 Suppose that ${\bf (N_0)}$ and ${\bf (N_3)}$ hold, and $d\in (d_{k_1,k_1+1},d_{k_1-1,k_1}),~k_1\in \mathbb{N}$, $k_2=0$. Then  Turing-Hopf bifurcation with Hopf-pitchfork type  occurs for
(\ref{3.1b}) with $\tau=\tau_{k_2},~\varepsilon=\varepsilon_*$ when $a_{111},~a_{123},~\mathrm{Re}b_{112},~\mathrm{Re}b_{223}\neq 0$, and $a_{111}\mathrm{Re}b_{223}-a_{123}\mathrm{Re}b_{112}\neq 0$. Moreover, the simplified planar system, which corresponds to normal form \eqref{third}, is

\begin{equation}\label{eq473-2}
\begin{array}{rlc}
\dot{r}=& r(\varepsilon_1(\alpha)+r^2+b_0z^2),\\
\dot{z}=& z(\varepsilon_2(\alpha)+c_0r^2+d_0z^2),
\end{array}
\end{equation}
where
$\varepsilon_1(\alpha)=\mathrm{Re}b_2(\alpha)\mathrm{sign}(\mathrm{Re}b_{223}),
~\varepsilon_2(\alpha)=a_1(\alpha)\mathrm{sign}(\mathrm{Re}b_{223})~
,~b_0=
\frac{\mathrm{Re}b_{112}}{|a_{111}|}\mathrm{sign}(\mathrm{Re}b_{223}),
~c_0=\frac{a_{123}}{|\mathrm{Re}b_{223}|}\mathrm{sign}(\mathrm{Re}b_{223}),
~d_0=\mathrm{sign}(a_{111}\mathrm{Re}b_{223})$.
 \end{theorem}
Based on \cite [\S7.5] {Guck1983}, by different signs of
$b_0,c_0,d_0,d_0-b_0c_0$ in Table 1, Eq.(\ref{eq473-2}) has twelve distinct types of unfoldings, which are twelve
  essentially distinct types of phase portraits and bifurcation diagrams. With the help of analysis in \cite [Section 4] {AJ}, the results  in \cite{Guck1983} can be directly applied to analyzing the
equation (\ref{eq473-2}). 

\begin{table}
\begin{center}
\begin{tabular}{|c|cccccccccccc|}
\hline
~~~case~~~ &$\mathrm{I}a$ &$ \mathrm{I}b$ & $\mathrm{II} $
&  $ \mathrm{III}$ &$ \mathrm{IV}a$&$ \mathrm{IV}b$& $ \mathrm{V}$
&$ \mathrm{VI}a$ &  $ \mathrm{VI}b$ & $ \mathrm{VII}a$ & $ \mathrm{VII}b$ &$ \mathrm{VIII}$\\
\hline $d_0$ & $+1$ & $+1$ & $+1$ & $+1$ & $+1$ & $+1$ & $-1$& $-1$& $-1$& $-1$& $-1$& $-1$\\
 $ b_0$ & $+$ & $+$& $+$& $-$& $-$& $-$& $+$& $+$& $+$& $-$ & $-$& $-$\\
$c_0$ & $+$& $+$& $-$& $+$& $-$& $-$& $+$& $-$& $-$& $+$& $+$& $-$\\
$ d_0-b_0c_0$ & $+$& $-$& $+$& $+$& $+$& $-$& $-$& $+$& $-$& $+$& $-$& $-$\\
\hline
\end{tabular}
\end{center}
\caption{The twelve unfoldings of \eqref{eq473-2}, see \cite {Guck1983}}\label{tab1}
\end{table}

%

By the corresponding bifurcation diagrams and phase portraits of \eqref{eq473-2}, we can answer following questions:
\begin{enumerate}
    \item On which side of the Turing bifurcation critical value does purely spatially periodic pattern (that is spatially inhomogeneous steady state solutions) appear? Is it stable?

    \item On which side of the Hopf bifurcation critical value  does temporally periodic pattern (that is spatially homogeneous or inhomogeneous periodic orbits ) appear? Is it stable?

    \item What kind of  mixed spatiotemporal periodic patterns will emerge, owing to the mode interaction between Turing and Hopf bifurcations?

\end{enumerate}

In next section, we will answer these three questions under the given system parameters.
%
\section{Spatiotemporal patterns with Turing-Hopf bifurcation}

The model \eqref{3.1b} has five parameters: $a,\ b,\ d,\ \varepsilon,\
\tau$.
We choose parameters:
\begin{align}
&a=0.1,\ b=0.9.\label{11}
\end{align}
The equilibrium point is
$(u_*,v_*)=(1,0.9)$, and condition
$${\bf (N_0)}\quad u_*^2>2u_*v_*-1>0.$$
is satisfied. By ${\bf (N_1)}$, ${\bf (N_2)}$, \eqref{eB} and \eqref{e0} we calculate that $\varepsilon_1= 0.1167$,
$$
\varepsilon_B(d)=\left\{\begin{array}{lll} &0.1167, &\mathrm{if} ~0<d\leq 0.5931,\\
&\frac{4}{5(\pi^2d+1)}, &\mathrm{if} ~~~~~~~d\geq 0.5931,
\end{array}\right.
$$
and
$$
\varepsilon_*=\varepsilon_*(k_1,d)=\frac{4dk_1^2\pi^2-5}{5dk_1^2\pi^2(dk_1^2\pi^2+1)},~~d\in (d_{k_1,k_1+1},~d_{k_1-1,k_1}),k_1\in \mathbb{N}.$$
In the following, we consider different values of wave number $k_1$ to reveal spatiotemporal patterns with different spatial frequencies.
\begin{figure}[htbp]
	\centering
	\includegraphics[scale=0.62]{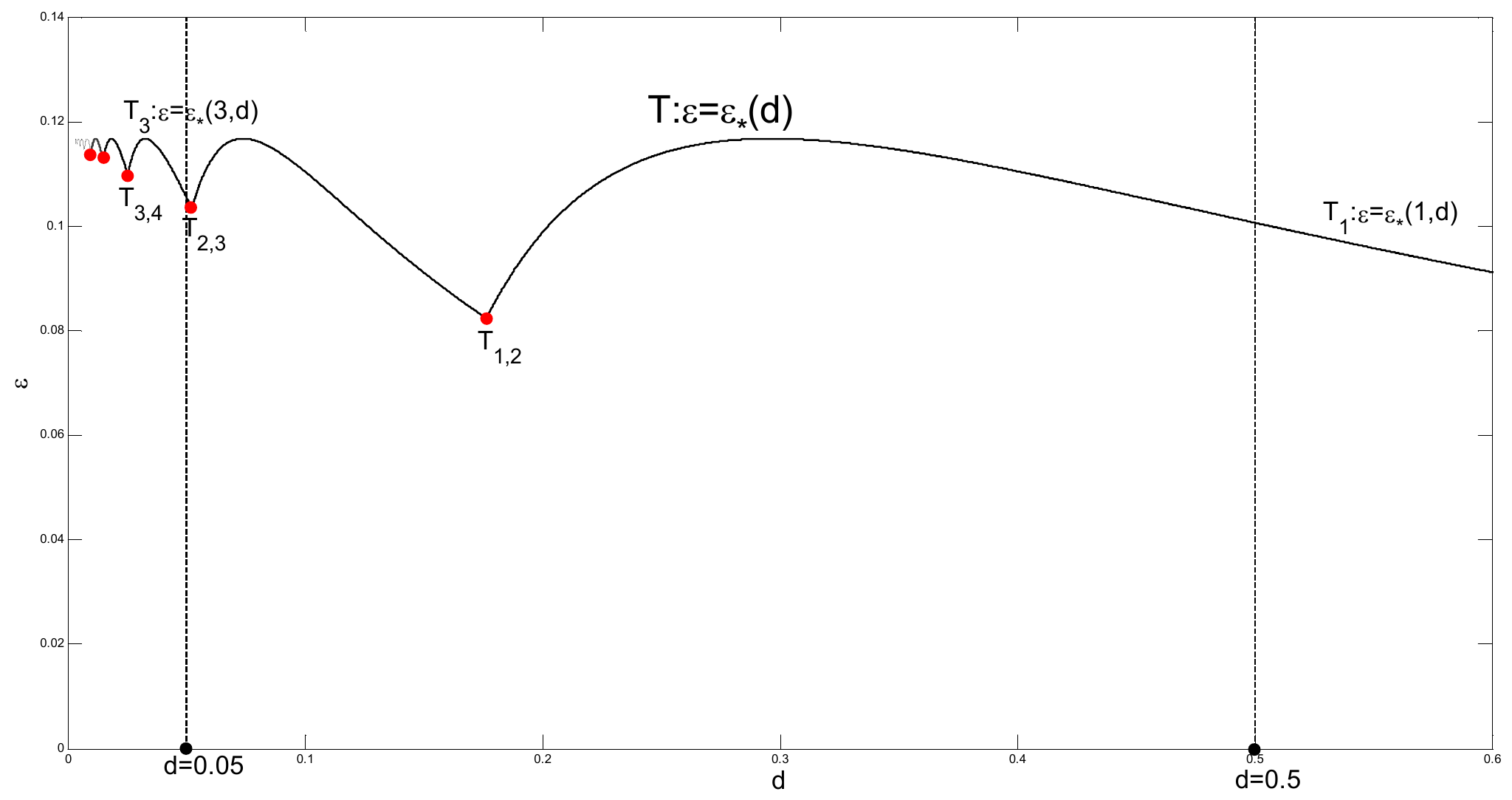}
	\caption{The line $d=0.5$ and $d=0.05$ intersect with Turing bifurcation line $T$ at $\varepsilon_*(1,0.5)=0.1007$ and $\varepsilon_*(3,0.05)=0.1056$, respectively.
	}
	\label{fig33}
\end{figure}

\begin{example}\label{exm:1}
Take $k_1=1$, then $d_{1,2}=0.1765$ by \eqref{dkk+}. Choose $d=0.5\in (d_{1,2},~+\infty)$, thus $\varepsilon_*=\varepsilon_*(1,0.5)=0.1007$. So, system (\ref{3.1b}) with $d=0.5$ undergoes $1-$mode Turing bifurcation near equilibrium $(1,0.9)$ at $\varepsilon=0.1007$. See Figure \ref{fig33}.

By $(u_*^2-2u_*v_*)^2(\varepsilon^2+1)<1$ for $\varepsilon>0.1007 $, we have $K^*=K^0$.
Furthermore, by (\ref{k0}), we obtain $K^0:=K^0(\varepsilon^*)=0.2721$. According to the first item of Theorem 2, we assert that equation (\ref{omegak}) has a positive root $\omega_k^+$ when $0\leq k<0.2721$.  
From Assumption ${\bf (N3)}$, we know that
$$k_2=0.$$
 By (\ref{omegakk}) and (\ref{imagi}), we obtain
$$\omega_0^+=0.9144,\ \tau_0=0.2171$$
Thus, by Theorem \ref{thm:2.12} and \ref{thm:2.15} we conclude that
\begin{corollary}
For parameters $a=0.1,\ b=0.9,\ d=0.5$, we have
\begin{enumerate}
\item [(1)]
 System (\ref {3.1b}) undergoes $(1,0)-$mode Turing-Hopf bifurcation near equilibrium $(u,v)=(1,0.9)$ at $\tau=0.2171,\ \varepsilon=0.1007$.
\item [(2)] The equilibrium $(u,v)=(1,0.9)$ is asymptotically stable in system (\ref{3.1b}) with $\tau\in [0,0.2171)$ for $ \varepsilon >0.1007$, and unstable for $0<\varepsilon<0.1007$.%
\end{enumerate}
\end{corollary}

Now, let's calculate the normal form. From (\ref {pqN}) and (\ref {cf}), we have
$$q_1^0=0.16837,~~q_2^0=0.55554-0.60759\mathrm{i},
~~p_1^0=-0.30307,~~p_2^0=-0.99997 + 1.0937\mathrm{i},$$
$$N_1=1.2192,~~N_2=1.0848 + 1.4353\mathrm{i},$$
$$r_{k_2}=0,~~q_{k_2}=1,~~p_{k_2}=1,$$
$$r_{k_1}=7.3859,~~q_{k_1}=-7.3859,~~p_{k_1}=6.4359,$$
$$r_{2k_1}=58.9574,~~q_{2k_1}=-32.5438,~~p_{2k_1}=22.7260.$$
By (\ref{eq03-3-h}), we obtain that
$$
\begin{array}{rll}
h_{200}^0(0)=&\left(\begin{array}{c}0.0084241\\-0.0079706
\end{array}
\right),\quad h_{200}(-1)&=\left(\begin{array}{c}-0.097808\\0.088523
\end{array}
\right)
\end{array},
$$
$$
\begin{array}{rl}
h_{200}^2(0)=&h_{200}^2(-1)=\left(\begin{array}{c}0.31035\\-0.046975
\end{array}
\right)
\end{array},
$$
$$
h_{011}^0(0)=\left(\begin{array}{c}-0.031531\\0.029834
\end{array}
\right),\quad
h_{011}^0(-1)=\left(\begin{array}{c}0.36609\\-0.33134
)
\end{array}\right),
$$
$$
h_{020}^0(0)=\left(\begin{array}{c}0.0020659 + 0.069984\mathrm{i}\\-0.0021214 - 0.066282\mathrm{i}
\end{array}
\right),\quad
h_{020}^0(-1)=\left(\begin{array}{c}-0.08515 - 0.80808\mathrm{i}\\0.082398 + 0.73077\mathrm{i}
\end{array}
\right),
$$
$$
h_{110}^1(0)=\left(\begin{array}{c}-0.0043069 + 0.10203\mathrm{i}\\0.0073407 - 0.25347\mathrm{i}
\end{array}
\right),\quad
h_{110}^1(-1)=\left(\begin{array}{c}0.03495 -0.24323 \mathrm{i}\\-0.048564 - 0.14565\mathrm{i}
\end{array}
\right),
$$
$$
h_{101}^1=\overline{h_{110}^1},\quad h_{002}^0=\overline{h_{020}^0},\quad h_{011}^2=0.
$$
Substituting above parameter values into expression (\ref {a1}), (\ref {eq03-3++}), 
the coefficients of normal form (\ref{third}) are obtained,
\begin{equation}\label{z123}
\begin{aligned}
&a_1(\alpha)=-0.00018873\alpha_1-0.8787\alpha_2,\\
&b_2(\alpha)=(0.07723 + 0.83252 \mathrm{i})\alpha_1,\\
&a_{11}=a_{23}=b_{12}=0,\\
&a_{111}= -0.1399,~~~~~~~b_{112}=-0.0906 + 0.0967\mathrm{i}, \\
&a_{123}=-0.1966,~~~~~~~b_{223}= -0.1675 - 0.0489\mathrm{i}.
\end{aligned}
\end{equation}
Thus, in corresponding planar system (\ref{eq473-2})
\begin{equation}\label{eq473-3}
\begin{array}{rlc}
\dot{r}=& r(\varepsilon_1(\alpha)+r^2+b_0z^2),\\
\dot{z}=& z(\varepsilon_2(\alpha)+c_0r^2+d_0z^2),
\end{array}
\end{equation}
we obtain that $\varepsilon_1(\alpha)=-0.07723 \alpha_1$, $\varepsilon_2(\alpha)=0.00018873\alpha_1+0.8787\alpha_2$, $b_0=0.6476,~c_0=1.1737,~d_0=1,~d_0-b_0c_0=0.2399$, and Case $\mathrm{Ia}$ in Table 1 occurs. Taking notice of $\mathrm{sign}(\mathrm{Re}b_{223})=-1$ in the coordinate transformation in \cite[Section 3]{JAS}, by phase portraits  in \cite [\S7.5] {Guck1983}, we show the complete bifurcation set for (\ref{3.1b}) with $a=0.1,\ b=0.9,\ d=0.5$ about parameters $(\tau,\varepsilon)$ and phase portraits for (\ref{eq473-3}) in Figure \ref{fig1}.
%

\begin{figure}[htbp]
	\centering
\subfigure[]{
		\begin{minipage}{0.4\linewidth}
			\centering
			\includegraphics[scale=0.30]{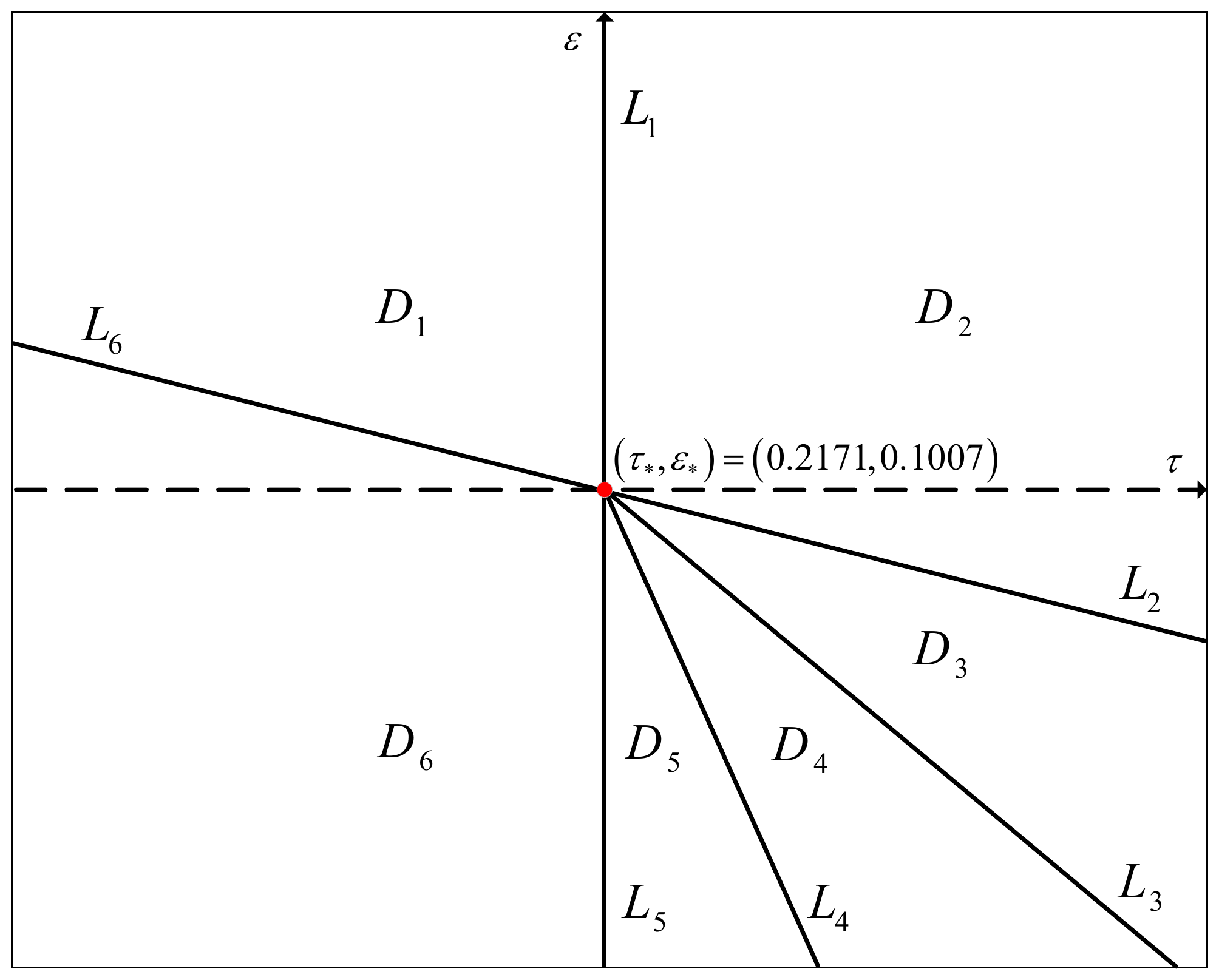}
	\end{minipage}}
\subfigure[]{
		\begin{minipage}{0.56\linewidth}
			\centering
			\includegraphics[scale=0.35]{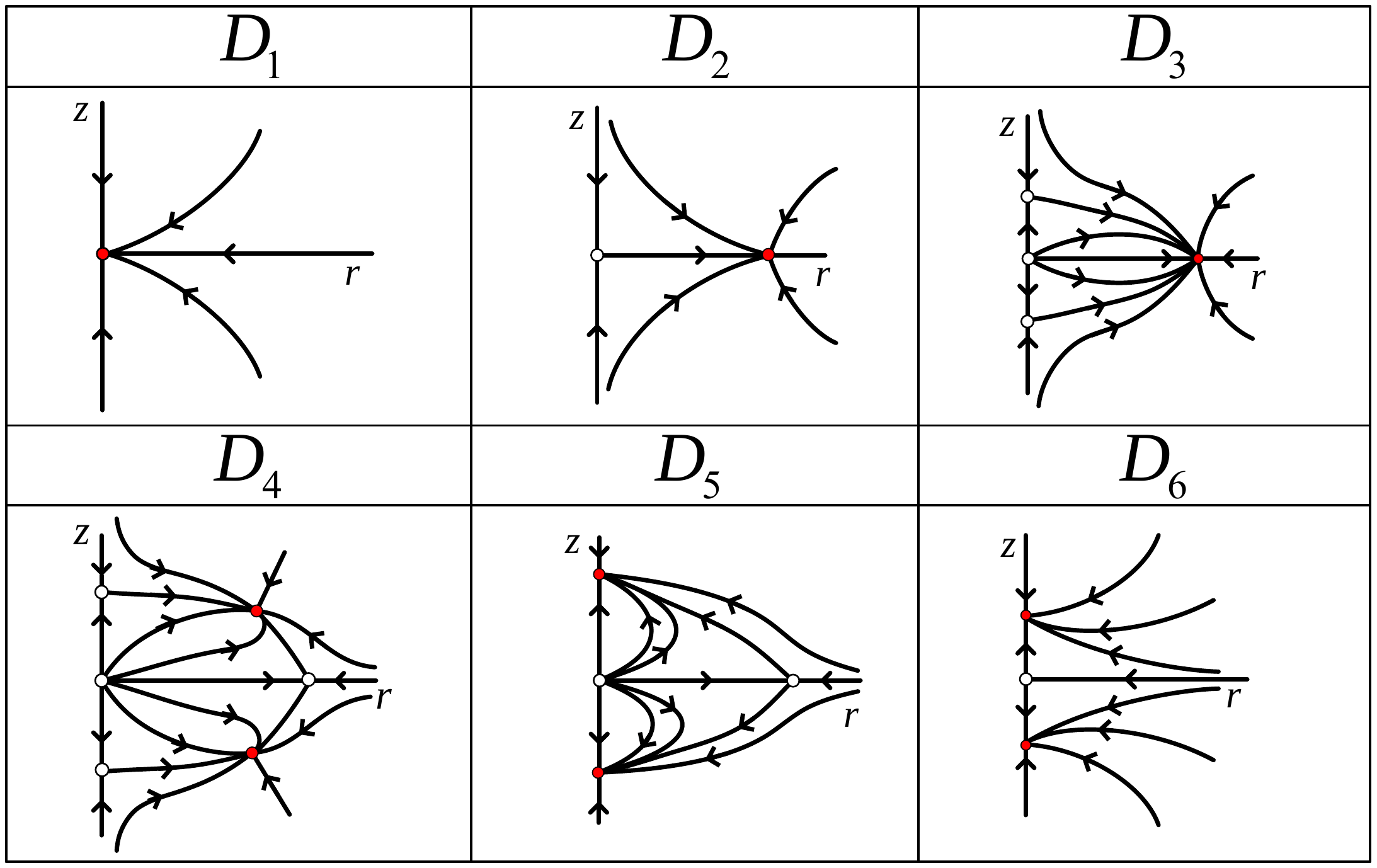}
		\end{minipage}}
\caption{In (a), bifurcation sets in $(\tau,\varepsilon)$ plane for (\ref{3.1b}) with $a=0.1,\ b=0.9,\ d=0.5$. In (b), phase portraits for (\ref{eq473-3}) for Case $ \mathrm{Ia}$. }\label{fig1}
\end{figure}

In Figure \ref{fig1}(a), critical bifurcation lines are, respectively,
$$
\begin{array}{rlc}
L_1:&\tau=\tau_*,~\varepsilon>\varepsilon_*,\\
L_2:&\varepsilon=\varepsilon_*-0.00021478(\tau-\tau_*),~\tau>\tau_*,\\
L_3:&\varepsilon=\varepsilon_*-0.1034(\tau-\tau_*),~\tau>\tau_*,\\
L_4:&\varepsilon=\varepsilon_*-0.1359(\tau-\tau_*),~\tau>\tau_*,\\
L_5:&\tau=\tau_*,~\varepsilon<\varepsilon_*,\\
L_6:&\varepsilon=\varepsilon_*-0.00021478(\tau-\tau_*),~\tau<\tau_*.
 \end{array}
   $$
By analysis in \cite[Section 4]{AJ}, we have following result.

\begin{theorem}\label{thm:4.1}
For system \eqref{3.1b} with $a=0.1,\ b=0.9,\ d=0.5$, dynamical phenomena are as follows when parameters $(\tau,\varepsilon)$ are sufficiently close to $(\tau_*,\varepsilon_*)$,
\begin{enumerate}
\item[(1)] The equilibrium $(u_*,v_*)$ is asymptotically stable when
$\varepsilon>\varepsilon_*-0.00021478(\tau-\tau_*)$
and $\tau<\tau_*$ (that is $(\tau, \varepsilon)\in D_1$);

 $0-$mode Hopf bifurcation occurs near the equilibrium $(u_*,v_*)$  at $(\tau,\ \varepsilon)\in L_1$.

\item[(2)] There exists an asymptotically stable spatially homogeneous periodic orbit which is bifurcated from equilibrium $(u_*,v_*)$, and $(u_*,v_*)$  loses its stability, when
$\varepsilon>\varepsilon_*-0.00021478(\tau-\tau_*)$
and $\tau>\tau_*$ (that is $(\tau, \varepsilon)\in D_2$);

 $1-$mode Turing bifurcation occurs near the equilibrium $(u_*,v_*)$  at $(\tau,\ \varepsilon)\in L_2$.

\item[(3)] There are two unstable spatially inhomogeneous steady state solutions  which are bifurcated from the equilibrium $(u_*,v_*)$ which is unstable, and the spatially homogeneous periodic orbit remains asymptotically stable, when
$\varepsilon_*-0.00021478(\tau-\tau_*)>\varepsilon>\varepsilon_*-0.1034(\tau-\tau_*)$
and $\tau>\tau_*$ (that is $(\tau, \varepsilon)\in D_3$);

 $1-$mode Turing bifurcation occurs near the spatially homogeneous periodic orbit at $(\tau,\ \varepsilon)\in L_3$.

 \item[(4)] There are two asymptotically stable spatially inhomogeneous periodic orbit  which are bifurcated from the spatially homogeneous periodic orbit. Moreover, their linear main parts are approximately
     $$E_*+\rho\phi_{2}(0)e^{\mathrm{i}\tau_*\omega_* t}+\bar{\rho}\bar{\phi}_{2}(0)e^{\mathrm{-i}\tau_*\omega_* t}\pm h\phi_{1}(0)\cos(\pi x),$$
     where $\rho$ and $h$ are some constants. The spatially homogeneous periodic orbit loses its stability, $(u_*,v_*)$ and two spatially inhomogeneous steady-state solutions are  still unstable, when
$\varepsilon_*-0.1359(\tau-\tau_*)<\varepsilon<\varepsilon_*-0.1034(\tau-\tau_*)$
and $\tau>\tau_*$ (that is $(\tau, \varepsilon)\in D_4$);

$0-$mode Hopf bifurcation occurs near two spatially inhomogeneous steady-state solutions at $(\tau,\ \varepsilon)\in L_4$.

 \item[(5)]  Two  spatially inhomogeneous periodic orbit disappear by Hopf bifurcation,  and two spatially inhomogeneous steady-state solutions are asymptotically stable, while $(u_*,v_*)$ and the spatially homogeneous periodic orbit are  still unstable, when
$\varepsilon<\varepsilon_*-0.1359(\tau-\tau_*)$
and $\tau>\tau_*$ (that is $(\tau, \varepsilon)\in D_5$);

$0-$mode Hopf bifurcation occurs near the equilibrium $(u_*,v_*)$  at $(\tau,\ \varepsilon)\in L_5$.

 \item[(6)] The spatially homogeneous periodic orbit disappears by Hopf bifurcation, and $(u_*,v_*)$ is unstable, while two  spatially inhomogeneous steady-state solutions remain asymptotically stable,
  when
$\varepsilon<\varepsilon_*-0.00021478(\tau-\tau_*)$
and $\tau<\tau_*$ (that is $(\tau, \varepsilon)\in D_6$);

$1-$mode Turing bifurcation occurs near the equilibrium $(u_*,v_*)$  at $(\tau,\ \varepsilon)\in L_6$.
 \end{enumerate}
where $\tau_*=0.2171$, $\varepsilon_*=0.1007$, $ \omega_*=0.9144$ and $(u_*,v_*)=(1,0.9)$.
\end{theorem}

\begin{remark}
$k-$mode Hopf (Turing) bifurcation means that the corresponding  $k-$th characteristic equation $D_{k}(\lambda)=0$ has a pair of pure imaginary roots (a zero toot).
\end{remark}

The above analytical information is an useful starting point for the use of adequate numerical tools. Now, by numerical simulations, we are going to show different temporal and spatial patterns of \eqref{3.1b} with $a=0.1,~b=0.9, d=0.5$ when $(\tau,\ \varepsilon) \in D_1-D_6$, respectively. We will see that numerical results are consistent with the theoretical results in Theorem \ref{thm:4.1}.
\begin{enumerate}
\item [(i)] Parameters $(\tau,\ \varepsilon)=(\tau_*,\ \varepsilon_*)+(-0.05,\ 0.05)\in
D_1$. Figure \ref{fig3} shows that the equilibrium $(u_*,v_*)$ is
asymptotically stable. $u(x,t),v(x,t)$ are solutions of \eqref{3.1b} with initial value functions $(\varphi(x,t),\psi(x,t))=(1+0.1\cos(\pi x),
1+0.1\cos (\pi x),~(x,t)\in[-0.1671,0]\times[0,1]$.
\begin{figure}[htbp]
	\centering
	\subfigure[$u(t,x)$]{
		\begin{minipage}{0.48\linewidth}
			\centering
            \includegraphics[scale=0.38]{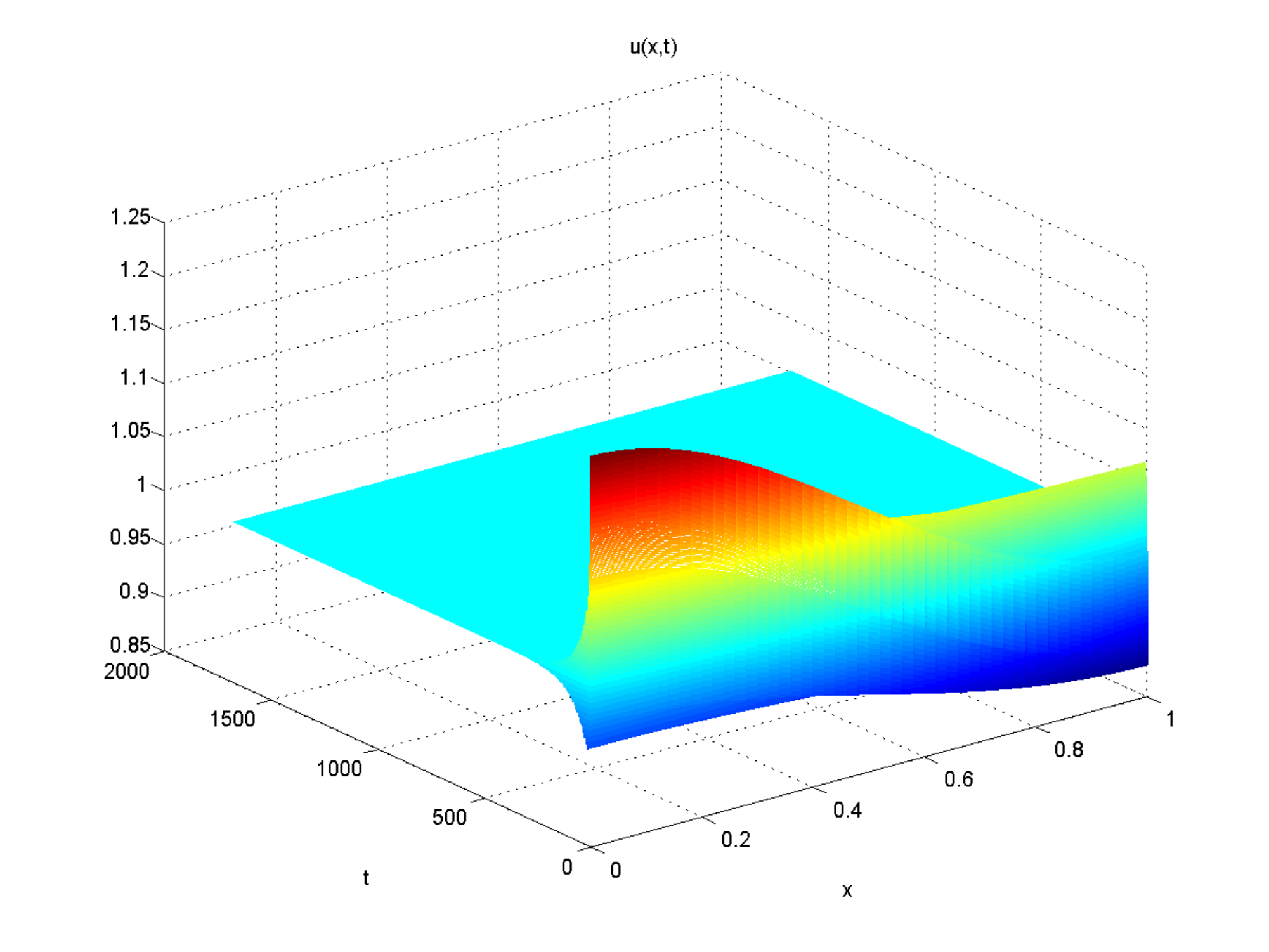}
	\end{minipage}}
	\subfigure[$v(t,x)$]{
		\begin{minipage}{0.48\linewidth}
			\centering
             \includegraphics[scale=0.38]{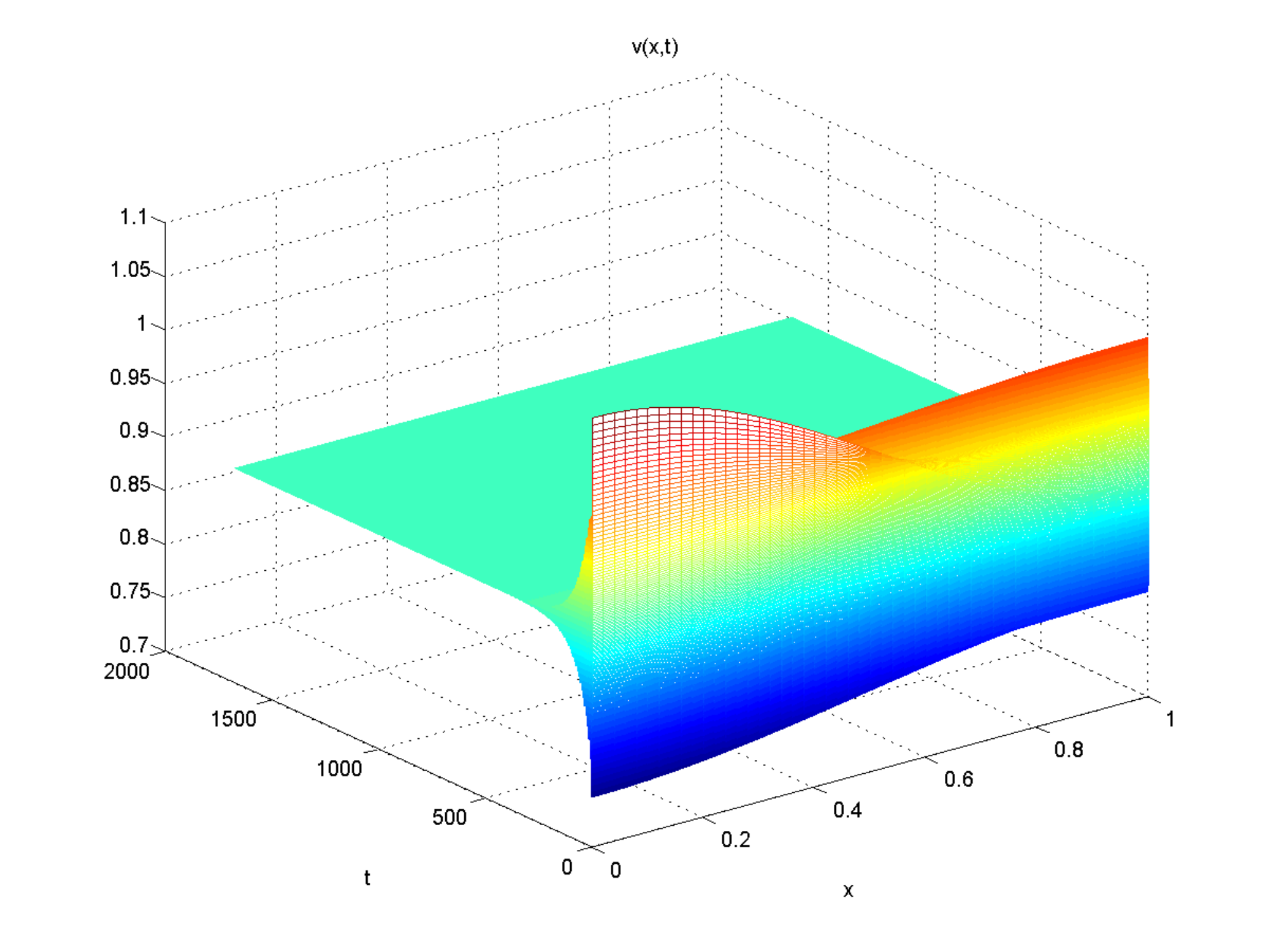}
		\end{minipage}
	}
	\caption{The equilibrium $(u_*,v_*)$ is
asymptotically stable when  $(\tau,\ \varepsilon)\in D_1$.} \label{fig3}
\end{figure}

\item[(ii)]Parameters $(\tau,\ \varepsilon)=(\tau_*,\ \varepsilon_*)+(0.05,\ 0.05)\in
D_2$. Figure \ref{fig4} shows that there exists an asymptotically stable spatially homogeneous periodic orbit in \eqref{3.1b}
. The initial value functions are $(\varphi(x,t),\psi(x,t))=(1+0.1\cos( \pi x),
1+0.1\cos(\pi x)),~(x,t)\in[-0.2671,0]\times[0,1]$, the simulated time is from $2630$ to $2671$.

For $(\tau,\ \varepsilon)\in
D_3$,  similar to Figure \ref{fig4}, an asymptotically stable spatially homogeneous periodic orbit  can be also  simulated, which is consistent with the assertions (3) of Theorem 4.1.
\begin{figure}[htbp]
	\centering
	\subfigure[$u(t,x)$]{
		\begin{minipage}{0.48\linewidth}
			\centering
            \includegraphics[scale=0.42]{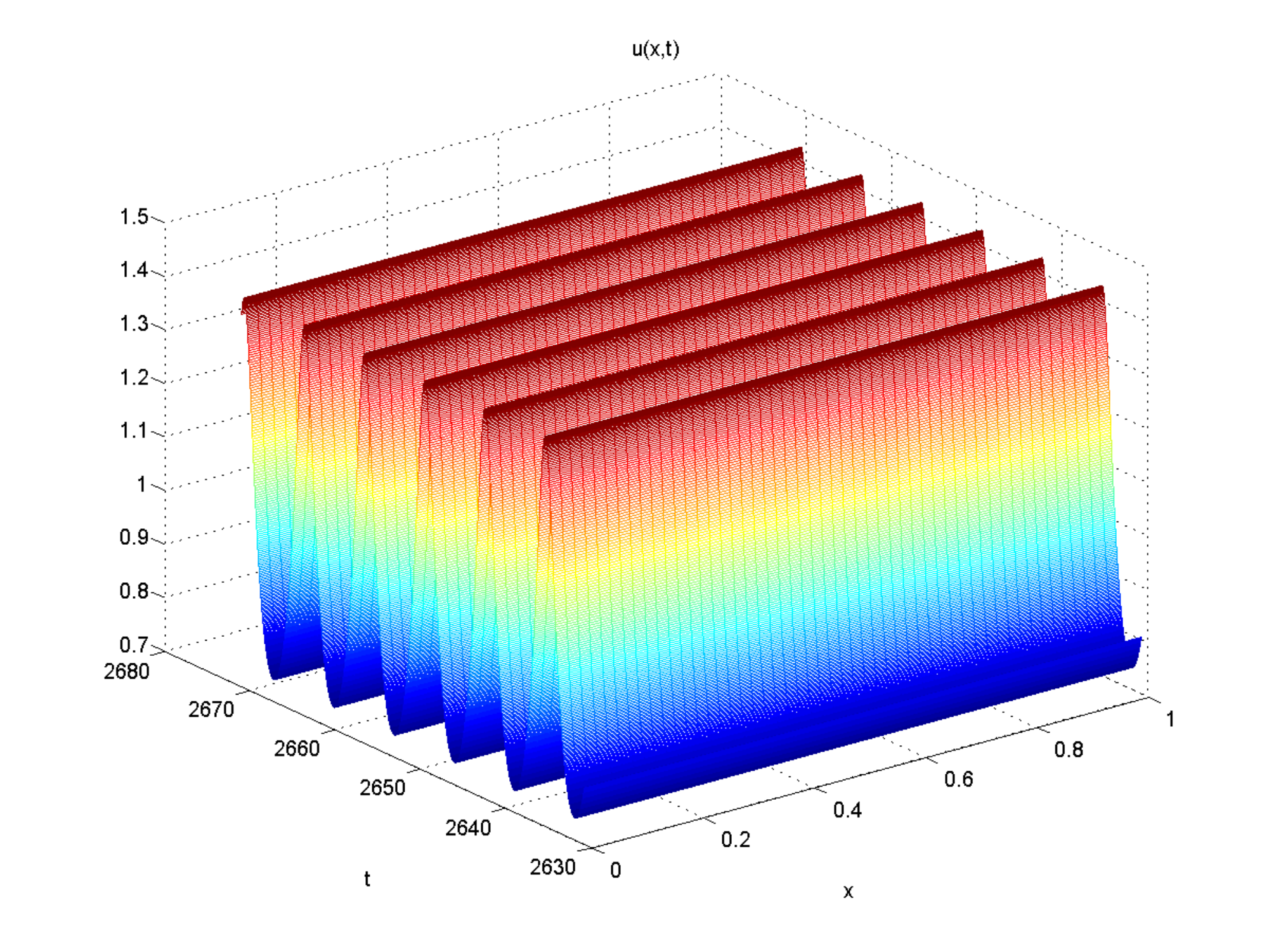}
	\end{minipage}}
	\subfigure[$v(t,x)$]{
		\begin{minipage}{0.48\linewidth}
			\centering
             \includegraphics[scale=0.42]{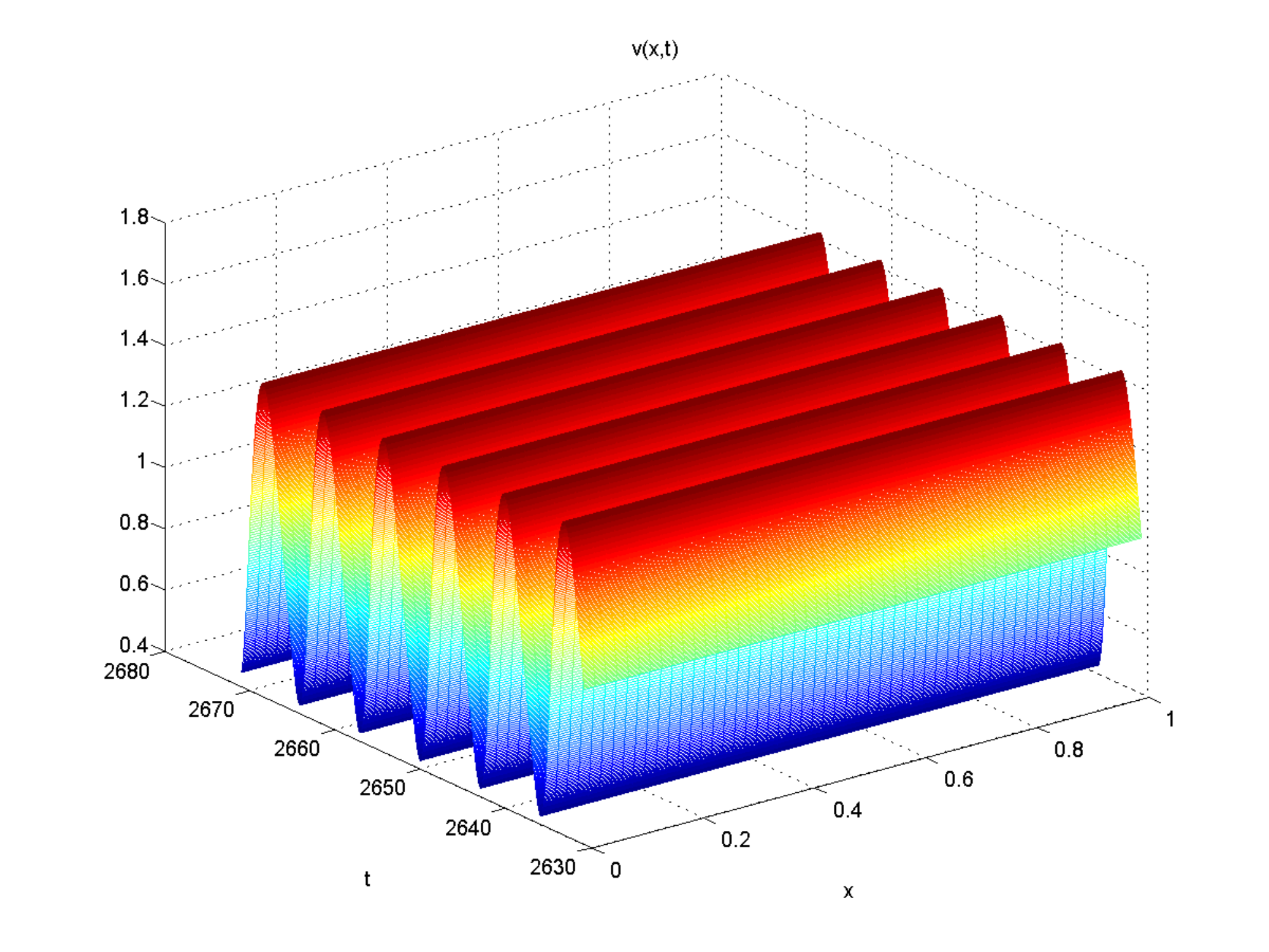}
		\end{minipage}
	}
	\caption{There exists an asymptotically stable spatially homogeneous periodic orbit in \eqref{3.1b} when  $(\tau,\ \varepsilon)\in D_2$.} \label{fig4}
\end{figure}

\item[(iii)]Parameters $(\tau,\ \varepsilon)=(\tau_*,\ \varepsilon_*)+(0.05,\ -0.0063)\in
D_4$. Figure \ref{fig5} shows that two stable spatially inhomogeneous periodic orbits coexist in \eqref{3.1b}. The initial value functions are $(\varphi(x,t),\psi(x,t))=(1-0.1\cos (\pi x),
1-0.1\cos (\pi x))$ in (a),(b) and $(\varphi(x,t),\psi(x,t))=(1+0.1\cos (\pi x),
1+0.1\cos (\pi x))$ in (c), (d), respectively, $(x,t)\in[-0.2671,0]\times[0,1]$,
the simulated time is from $2630$ to $2671$.
\begin{figure}[htbp]
	\centering
	\subfigure[$u(t,x)$]{
		\begin{minipage}{0.48\linewidth}
			\centering
            \includegraphics[scale=0.42]{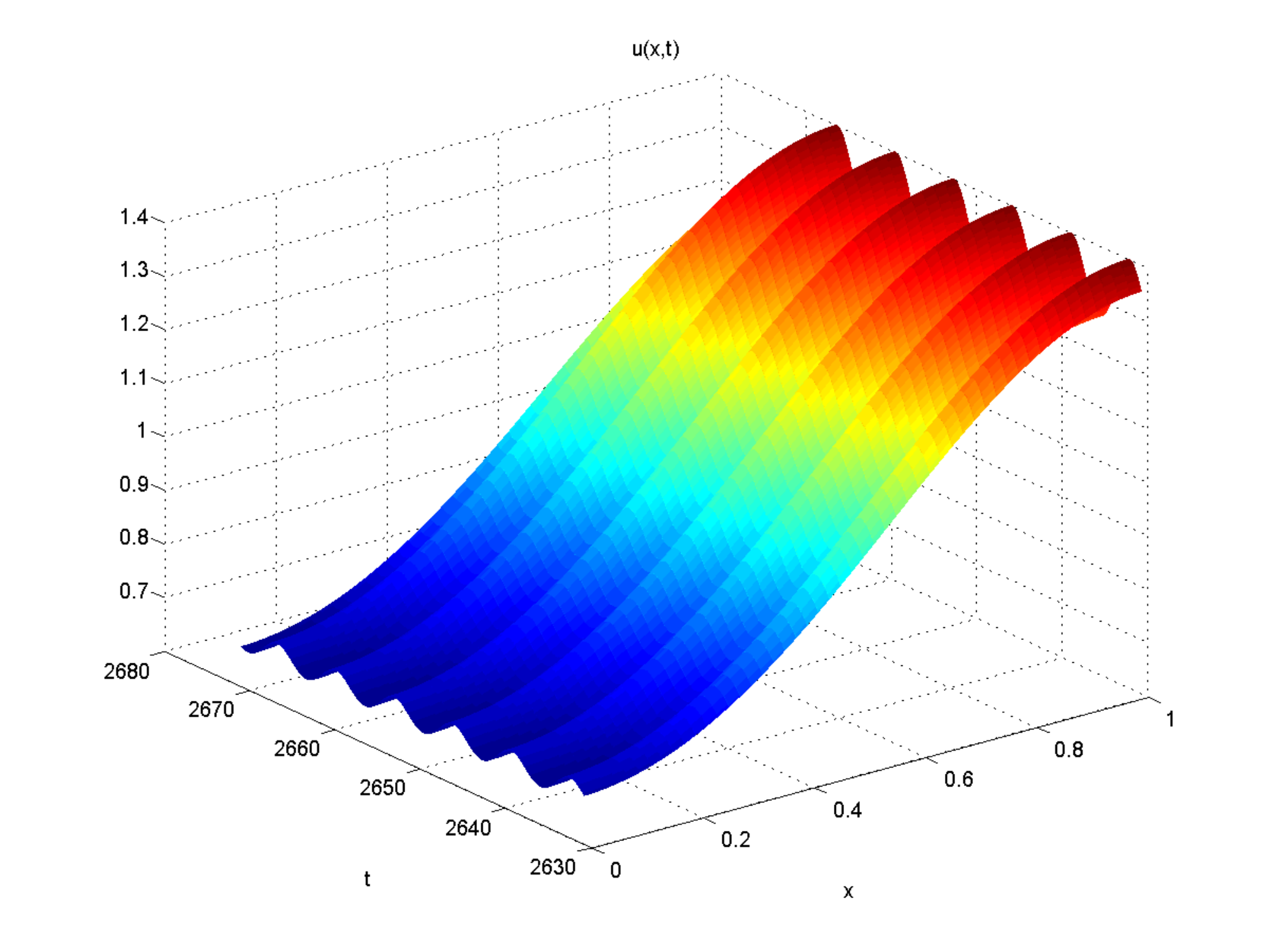}
	\end{minipage}}
	\subfigure[$v(t,x)$]{
		\begin{minipage}{0.48\linewidth}
			\centering
             \includegraphics[scale=0.42]{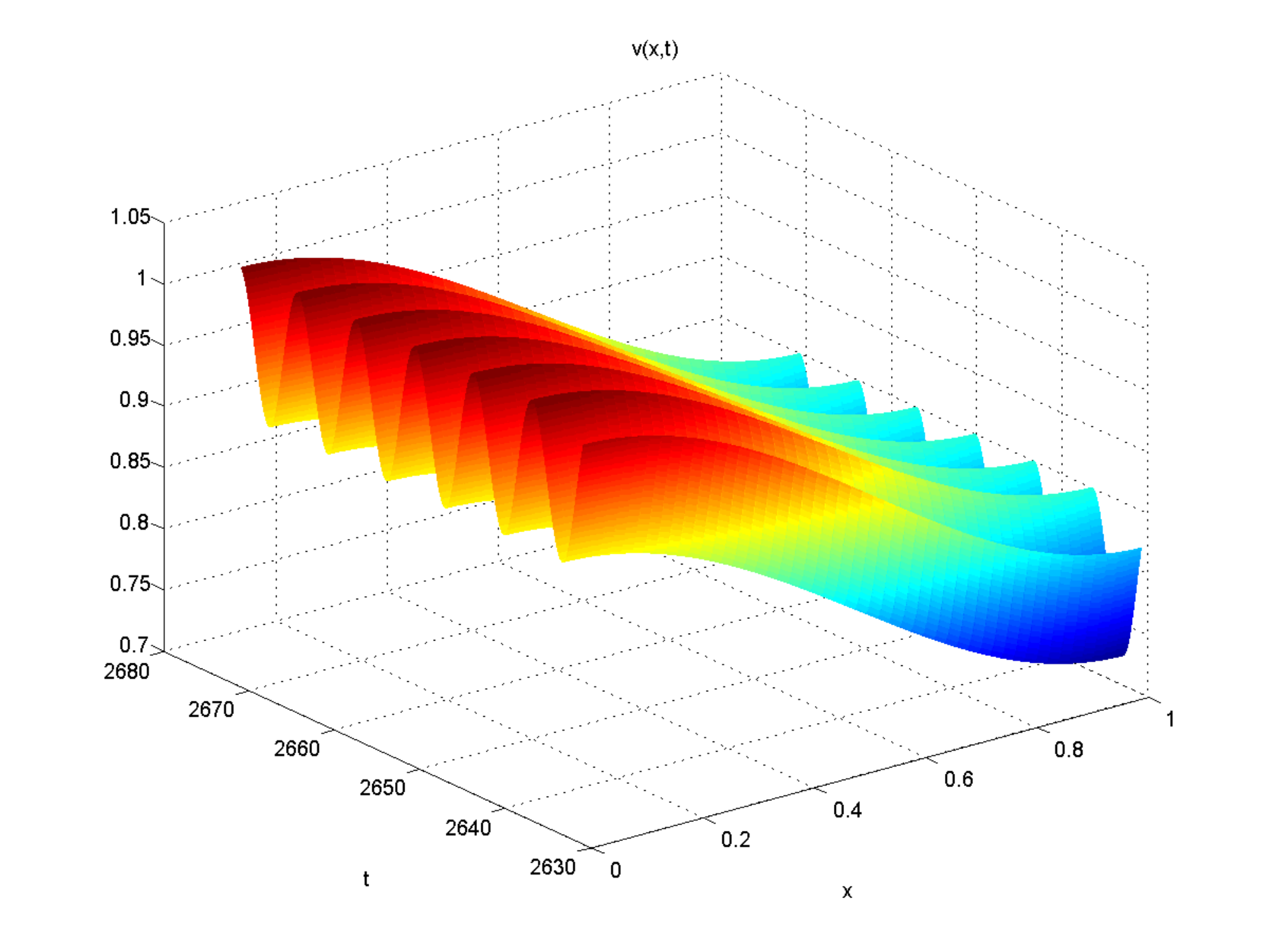}
		\end{minipage}
	}

	\subfigure[$u(t,x)$]{
		\begin{minipage}{0.48\linewidth}
			\centering
			\includegraphics[scale=0.42]{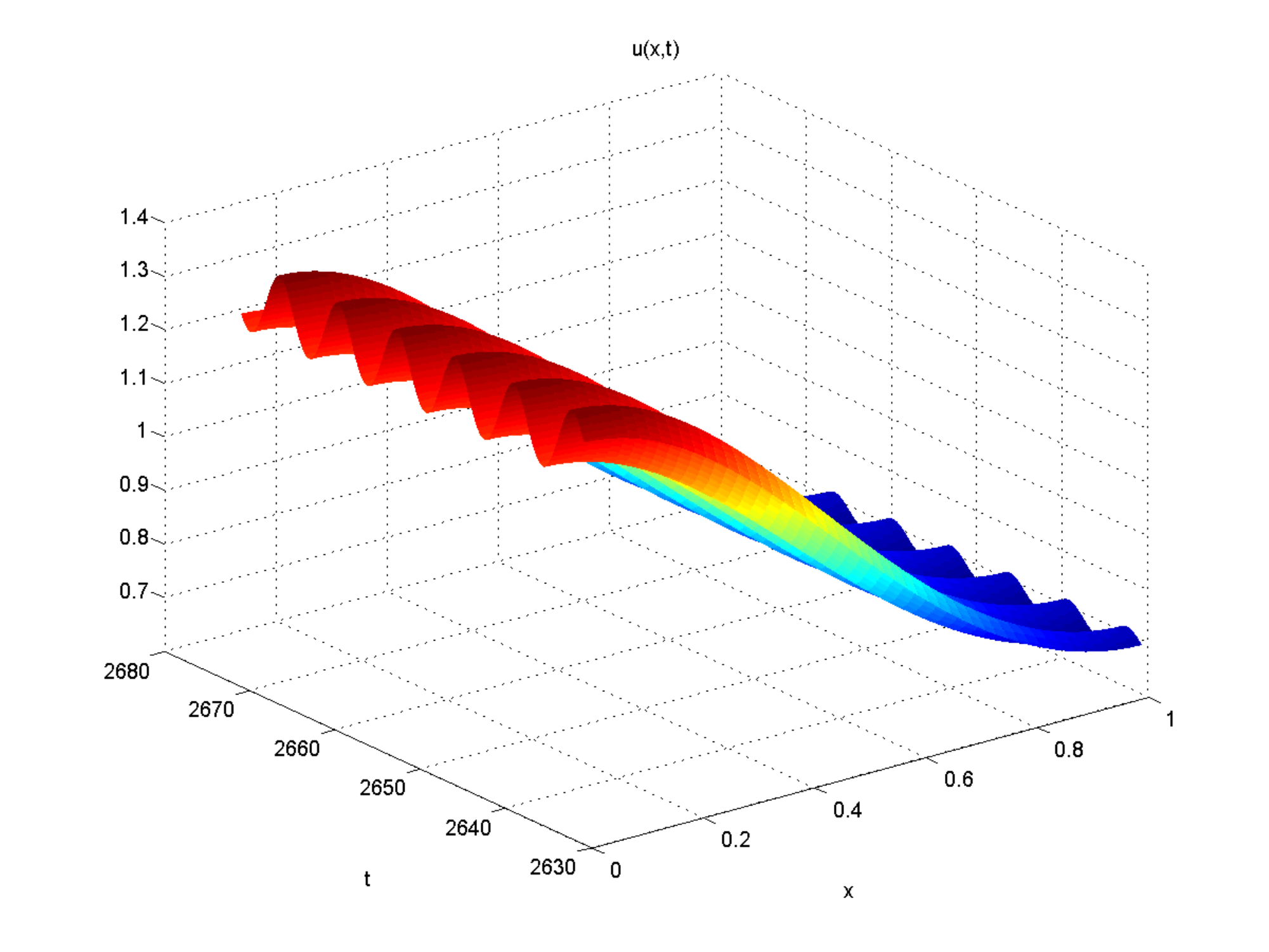}
	\end{minipage}}
	\subfigure[$v(t,x)$]{
		\begin{minipage}{0.48\linewidth}
			\centering
			\includegraphics[scale=0.42]{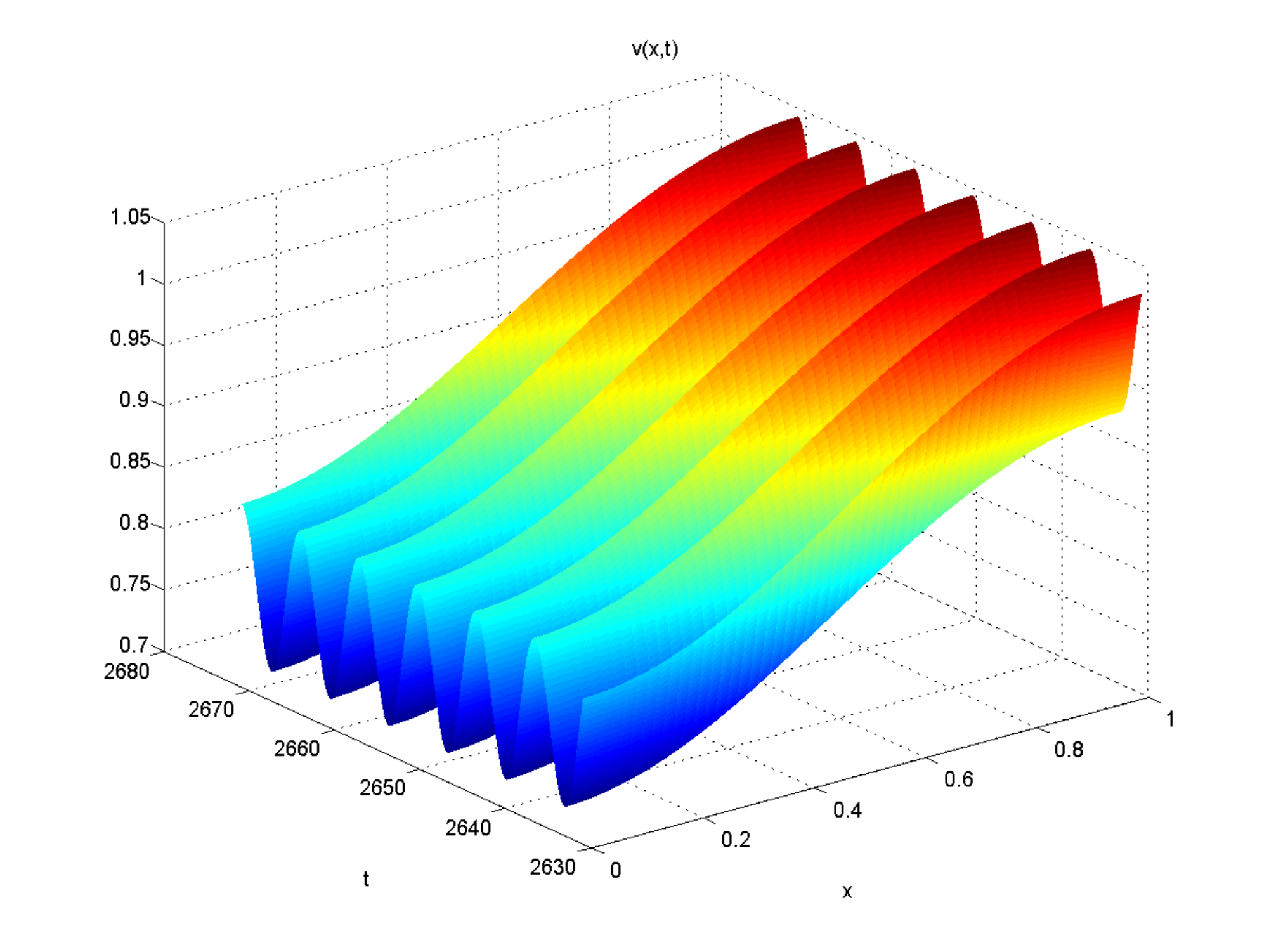}
		\end{minipage}
	}
	\caption{Two stable spatially inhomogeneous periodic orbits coexist in \eqref{3.1b} when  $(\tau,\ \varepsilon)\in D_4$.} \label{fig5}
\end{figure}

\item[(iv)]Parameters $(\tau,\ \varepsilon)=(\tau_*,\ \varepsilon_*)+(0.05,\ -0.03)\in
D_5$. Figure \ref{fig6} shows that two stable spatially inhomogeneous steady state solutions coexist in \eqref{3.1b}. The initial value functions are $(\varphi(x,t),\psi(x,t))=(0.9,
1.1)$ in (a),(b) and $(\varphi(x,t),\psi(x,t))=(1+0.1\cos (\pi x),
1+0.1\cos (\pi x))$ in (c), (d), respectively, $(x,t)\in[-0.2671,0]\times[0,1]$.
 The simulated time is from $1259$ to $1335$ in (a),(b) and from $2630$ to $2671$ in (c),(d), respectively.

For $(\tau,\ \varepsilon)\in
D_6$,  similar to Figure \ref{fig6}, a pair of stable spatially inhomogeneous steady state solutions can be also simulated, which are consistent with the assertions (6) of Theorem \ref{thm:4.1}.
\begin{figure}[htbp]
	\centering
	\subfigure[$u(t,x)$]{
		\begin{minipage}{0.48\linewidth}
			\centering
            \includegraphics[scale=0.42]{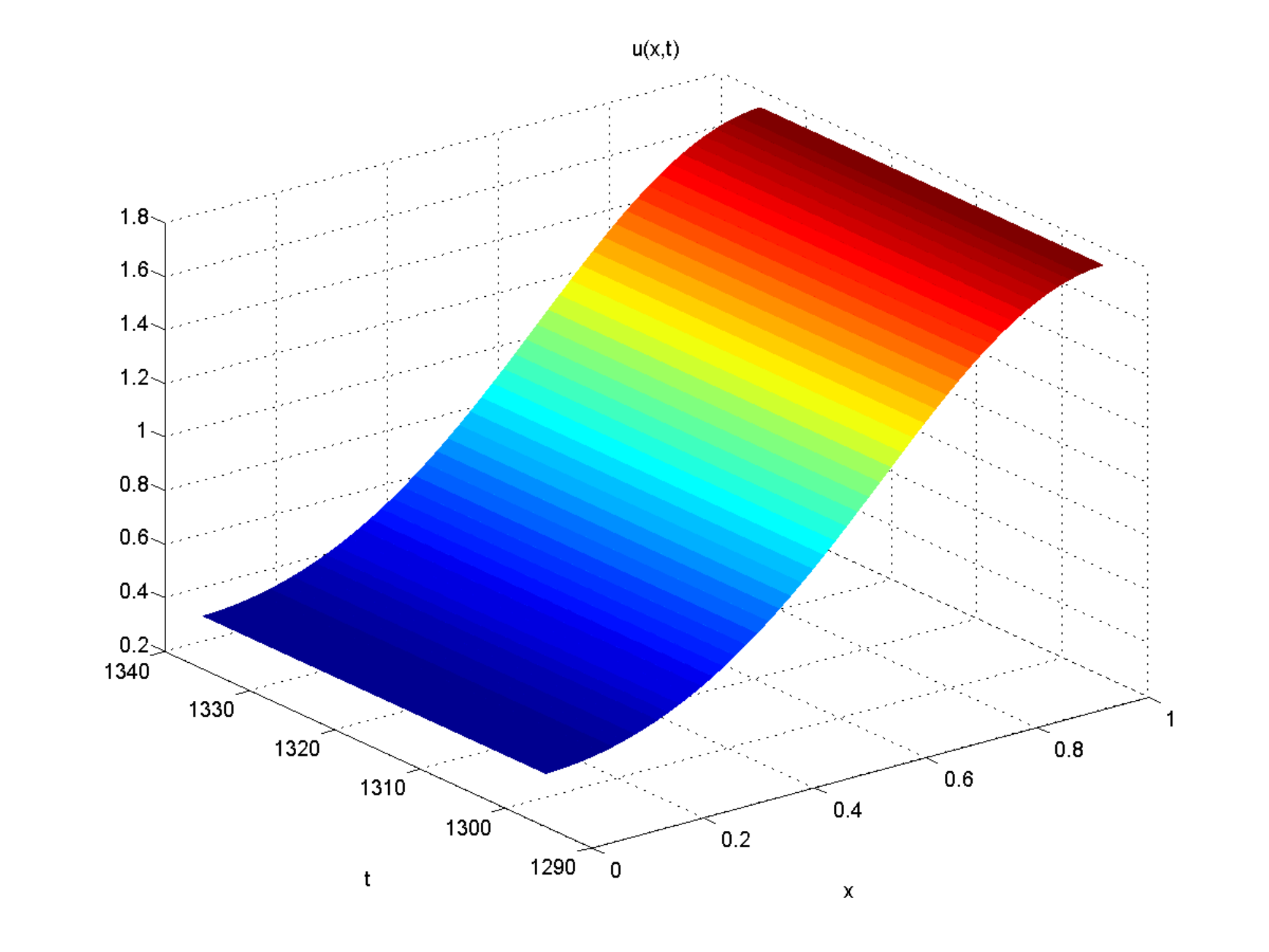}
	\end{minipage}}
	\subfigure[$v(t,x)$]{
		\begin{minipage}{0.48\linewidth}
			\centering
             \includegraphics[scale=0.42]{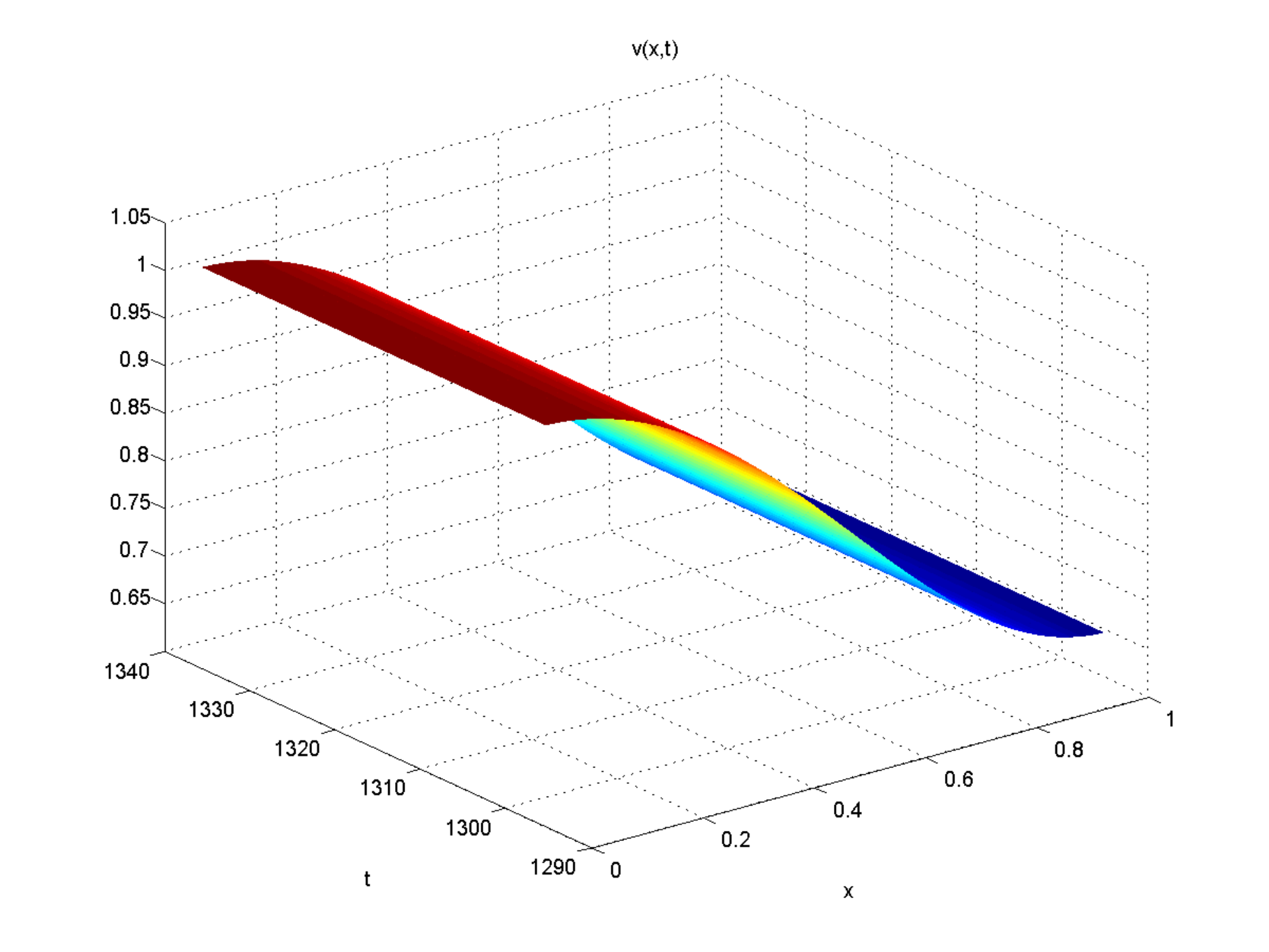}
		\end{minipage}
	}

	\subfigure[$u(t,x)$]{
		\begin{minipage}{0.48\linewidth}
			\centering
			\includegraphics[scale=0.42]{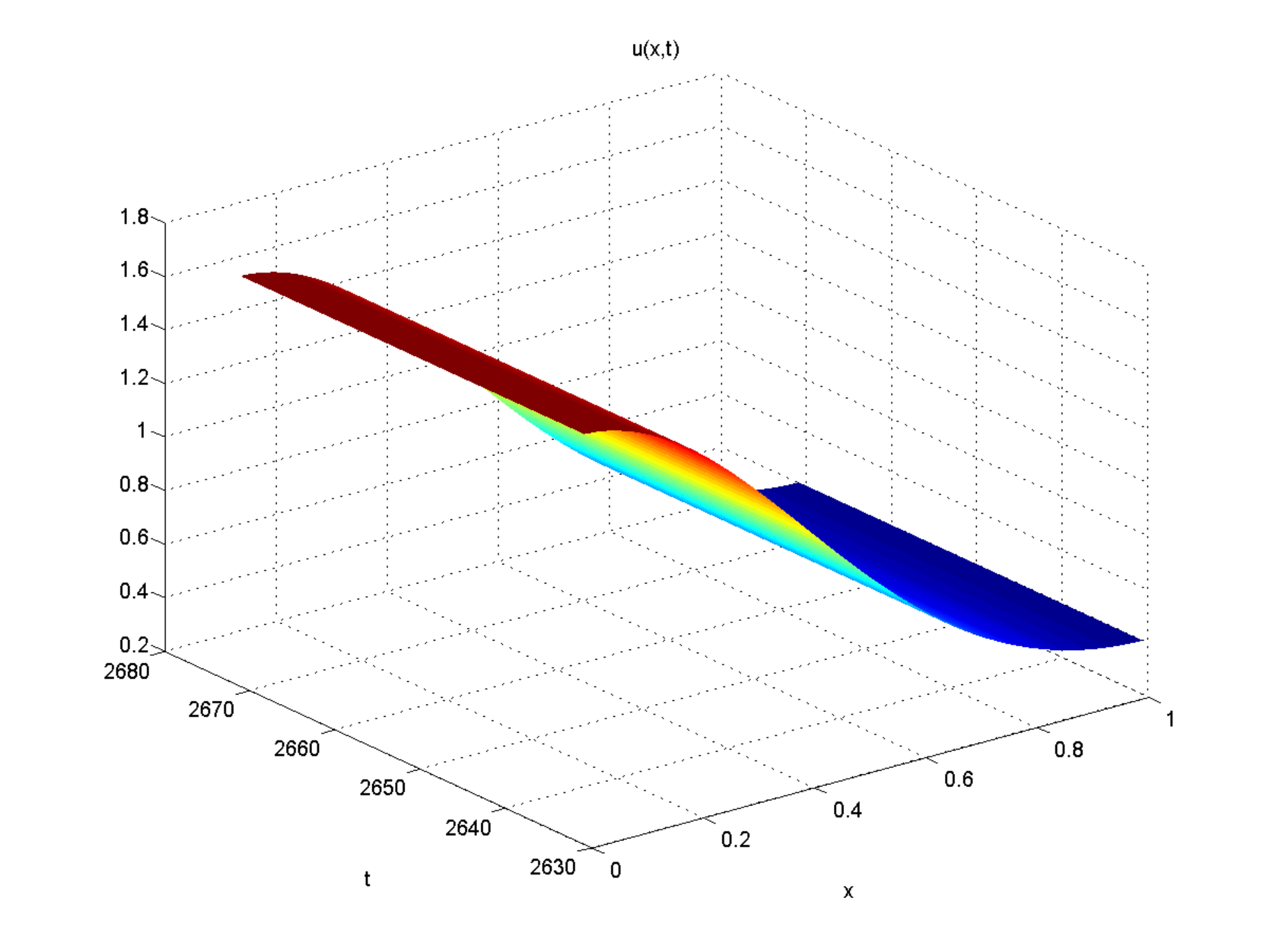}
	\end{minipage}}
	\subfigure[$v(t,x)$]{
		\begin{minipage}{0.48\linewidth}
			\centering
			\includegraphics[scale=0.42]{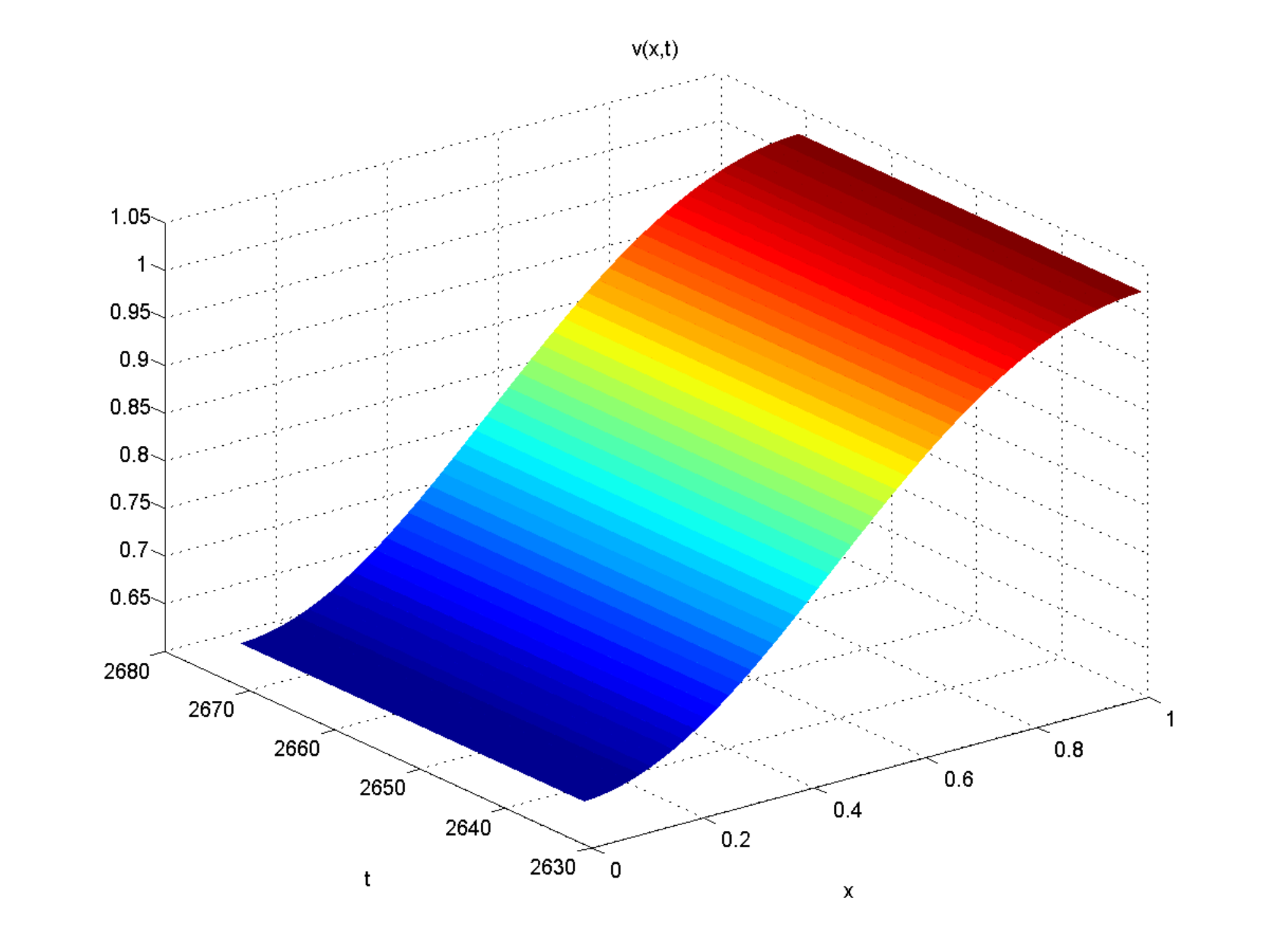}
		\end{minipage}
	}
	\caption{Two stable spatially inhomogeneous steady state solutions coexist  in \eqref{3.1b} when  $(\tau,\ \varepsilon)\in D_5$.} \label{fig6}
\end{figure}
\end{enumerate}

\end{example}

\begin{example}\label{exm:2}
Take $k_1=3$, then $d_{3,4}=0.0255,~d_{2,3}=0.0525$ by \eqref{dkk+}. Choose $d=0.05\in (d_{3,4},~d_{2,3})$, thus $\varepsilon_*=\varepsilon_*(3,0.05)=0.1056$, see Figure \ref{fig33}. So, the system (\ref{3.1b}) with $d=0.05$ undergoes $3-$mode Turing bifurcation near the equilibrium $(1,0.9)$ at $\varepsilon=0.1056$. Similar to previous discussion, $k_2=0$, $\omega_0^+=0.9144,\ \tau_0=0.2171$ are the same as Example \ref{exm:1}. Thus, we have following conclusions.
\begin{corollary}
For parameters $a=0.1,\ b=0.9,\ d=0.05$, we have that
\begin{enumerate}
\item[(1)] System (\ref {3.1b}) undergoes $(3,0)-$mode Turing-Hopf bifurcation near equilibrium $(u,v)=(1,0.9)$  at $\tau=0.2171,\ \varepsilon=0.1056$.
\item[(2)] The equilibrium $(u,v)=(1,0.9)$ is asymptotically stable in system (\ref{3.1b}) with $\tau\in [0,0.2171)$ for $ 0.1056<\varepsilon<0.1079 $, and unstable for $0.0600<\varepsilon<0.1056$.%
\end{enumerate}
\end{corollary}
Similarly, in corresponding  planar system (\ref{eq473-3}), the coefficients are $\varepsilon_1(\alpha)=-0.07723 \alpha_1$, $\varepsilon_2(\alpha)=0.00018873\alpha_1+0.8787\alpha_2$, $b_0=0.6476,~c_0=1.1737,~d_0=1,~d_0-b_0c_0=0.2399$, $\mathrm{sign}(\mathrm{Re}b_{223})=-1$. The Case $\mathrm{Ia}$ in Table 1 occurs again.
In Figure \ref{fig1}(a), critical bifurcation lines are, respectively,
$$
\begin{array}{rlc}
L_1:&\tau=\tau_*,~\varepsilon>\varepsilon_*,\\
L_2:&\varepsilon=\varepsilon_*+0.00020177(\tau-\tau_*),~\tau>\tau_*,\\
L_3:&\varepsilon=\varepsilon_*-0.1154(\tau-\tau_*),~\tau>\tau_*,\\
L_4:&\varepsilon=\varepsilon_*-0.1231(\tau-\tau_*),~\tau>\tau_*,\\
L_5:&\tau=\tau_*,~\varepsilon<\varepsilon_*,\\
L_6:&\varepsilon=\varepsilon_*+0.00020177(\tau-\tau_*),~\tau<\tau_*.
 \end{array}
   $$

   \begin{enumerate}
   \item[(iii)]Parameters $(\tau,\ \varepsilon)=(\tau_*,\ \varepsilon_*)+(0.05,\ -0.0063)\in
D_4$. Figure \ref{fig7} shows that two stable spatially inhomogeneous periodic orbits coexist in \eqref{3.1b}. The initial value functions are $(\varphi(x,t),\psi(x,t))=(1-0.1\cos (\pi x),
1-0.1\cos (\pi x))$ in (a),(b) and $(\varphi(x,t),\psi(x,t))=(1+0.1\cos (\pi x),
1+0.1\cos (\pi x))$ in (c), (d), respectively, $(x,t)\in[-0.2671,0]\times[0,1]$,
the simulated time is from $2630$ to $2671$.

\begin{figure}[htbp]
	\centering
	\subfigure[$u(t,x)$]{
		\begin{minipage}{0.48\linewidth}
			\centering
            \includegraphics[scale=0.42]{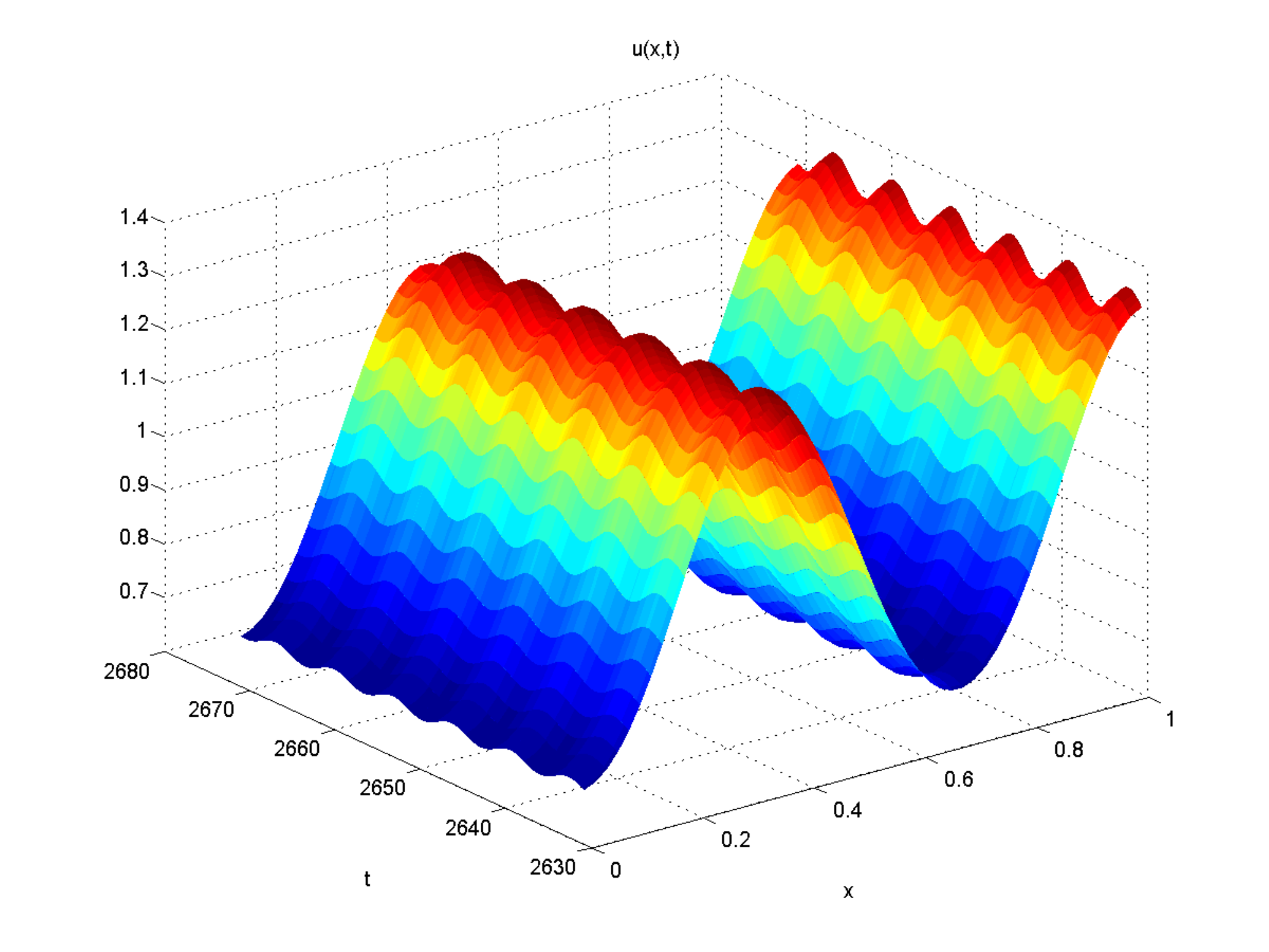}
	\end{minipage}}
	\subfigure[$v(t,x)$]{
		\begin{minipage}{0.48\linewidth}
			\centering
             \includegraphics[scale=0.42]{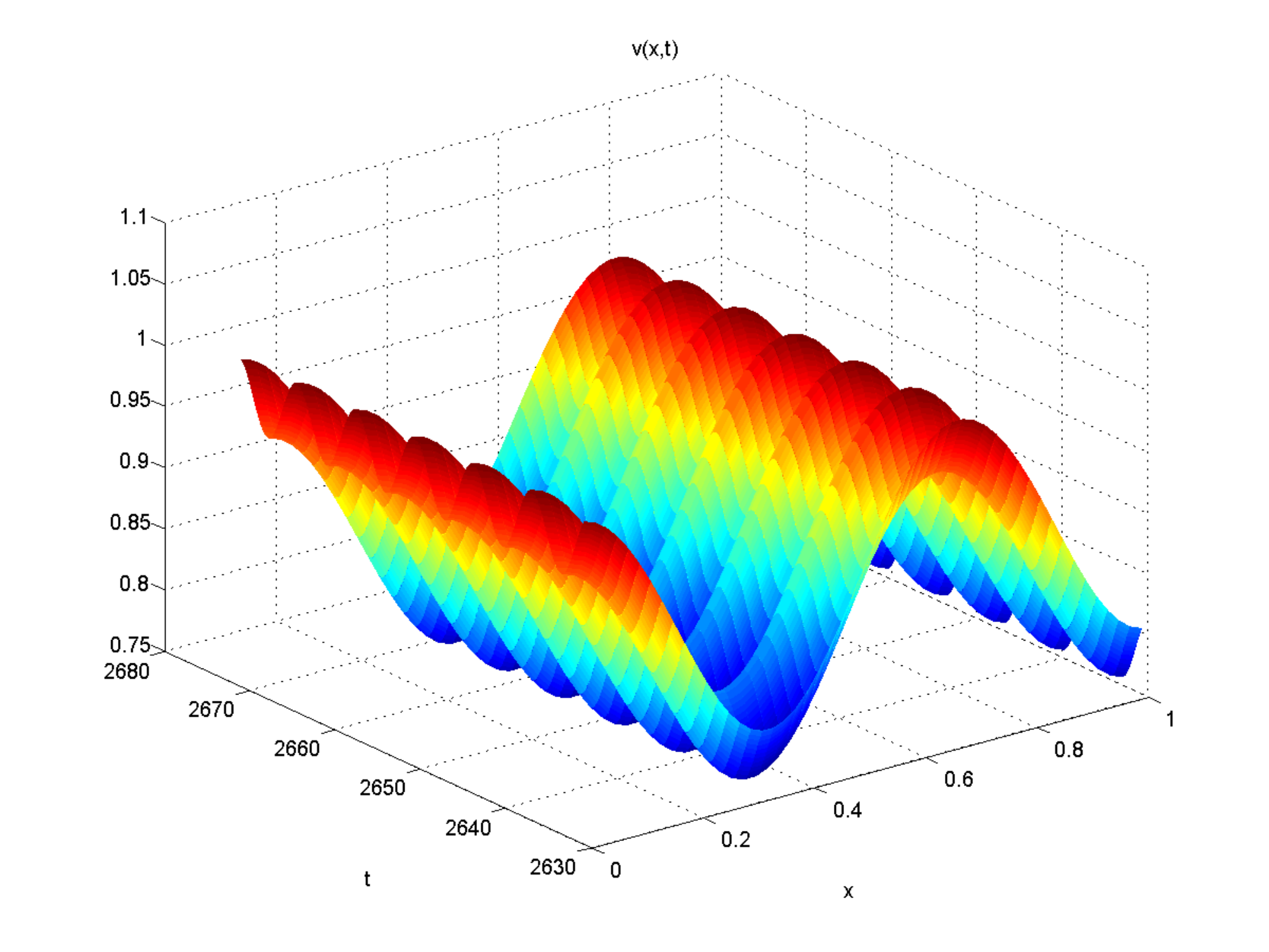}
		\end{minipage}
	}

	\subfigure[$u(t,x)$]{
		\begin{minipage}{0.48\linewidth}
			\centering
			\includegraphics[scale=0.42]{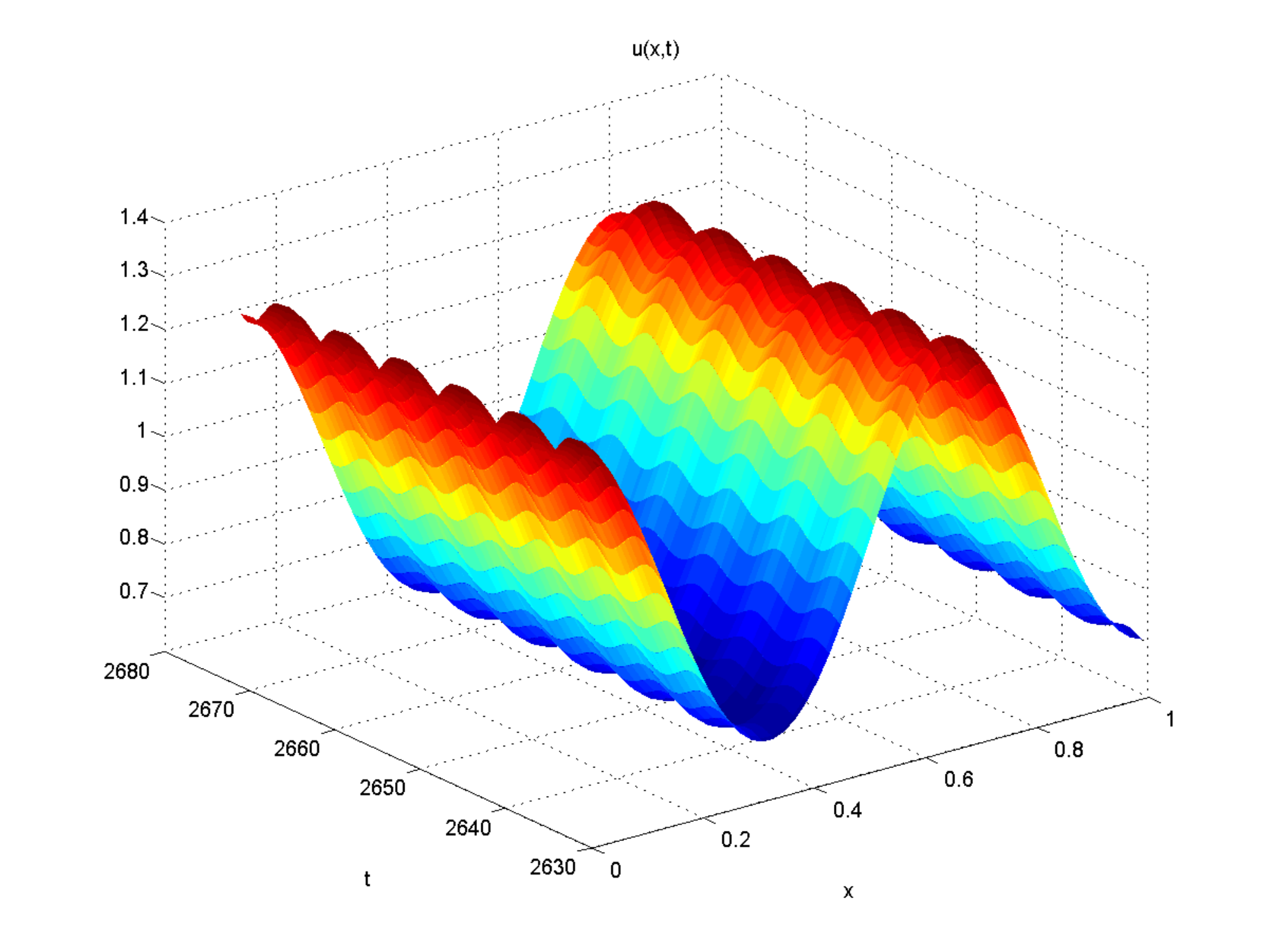}
	\end{minipage}}
	\subfigure[$v(t,x)$]{
		\begin{minipage}{0.48\linewidth}
			\centering
			\includegraphics[scale=0.42]{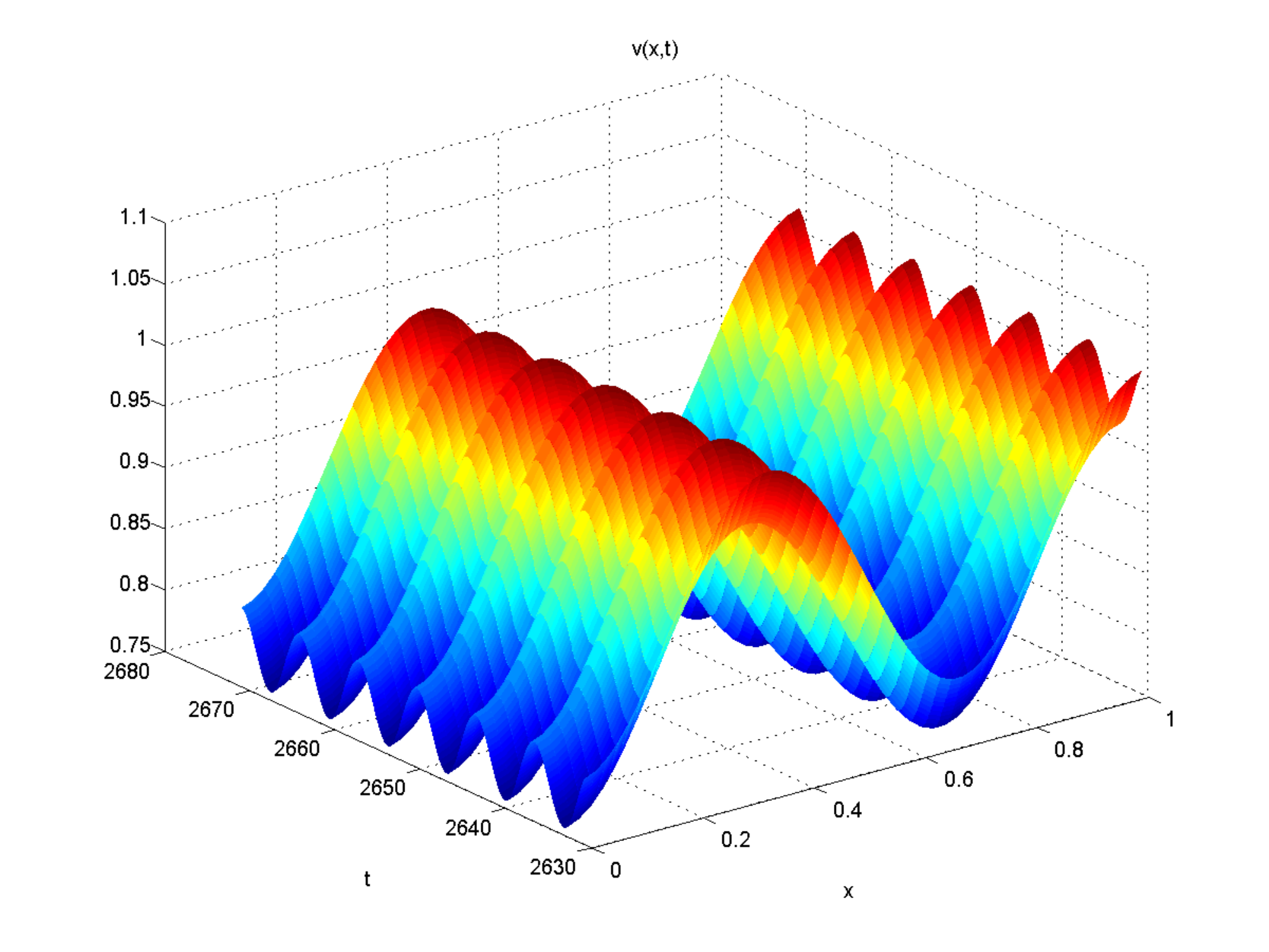}
		\end{minipage}
	}
	\caption{Two stable spatially inhomogeneous periodic orbits coexist for \eqref{3.1b} when  $(\tau,\ \varepsilon)\in D_4$.} \label{fig7}
\end{figure}

\item[(iv)]Parameters $(\tau,\ \varepsilon)=(\tau_*,\ \varepsilon_*)+(0.05,\ -0.03)\in
D_5$. Figure \ref{fig8} shows that two stable spatially inhomogeneous steady state solutions coexist for \eqref{3.1b}. The initial value functions are $(\varphi(x,t),\psi(x,t))=(0.9,
1.1)$ in (a),(b) and $(\varphi(x,t),\psi(x,t))=(1+0.1\cos (\pi x),
1+0.1\cos (\pi x))$ in (c), (d), respectively, $(x,t)\in[-0.2671,0]\times[0,1]$.
The simulated time is from $3966$ to $4006$ in (a),(b),(c),(d).
%
\begin{figure}[htbp]
	\centering
	\subfigure[$u(t,x)$]{
		\begin{minipage}{0.48\linewidth}
			\centering
            \includegraphics[scale=0.42]{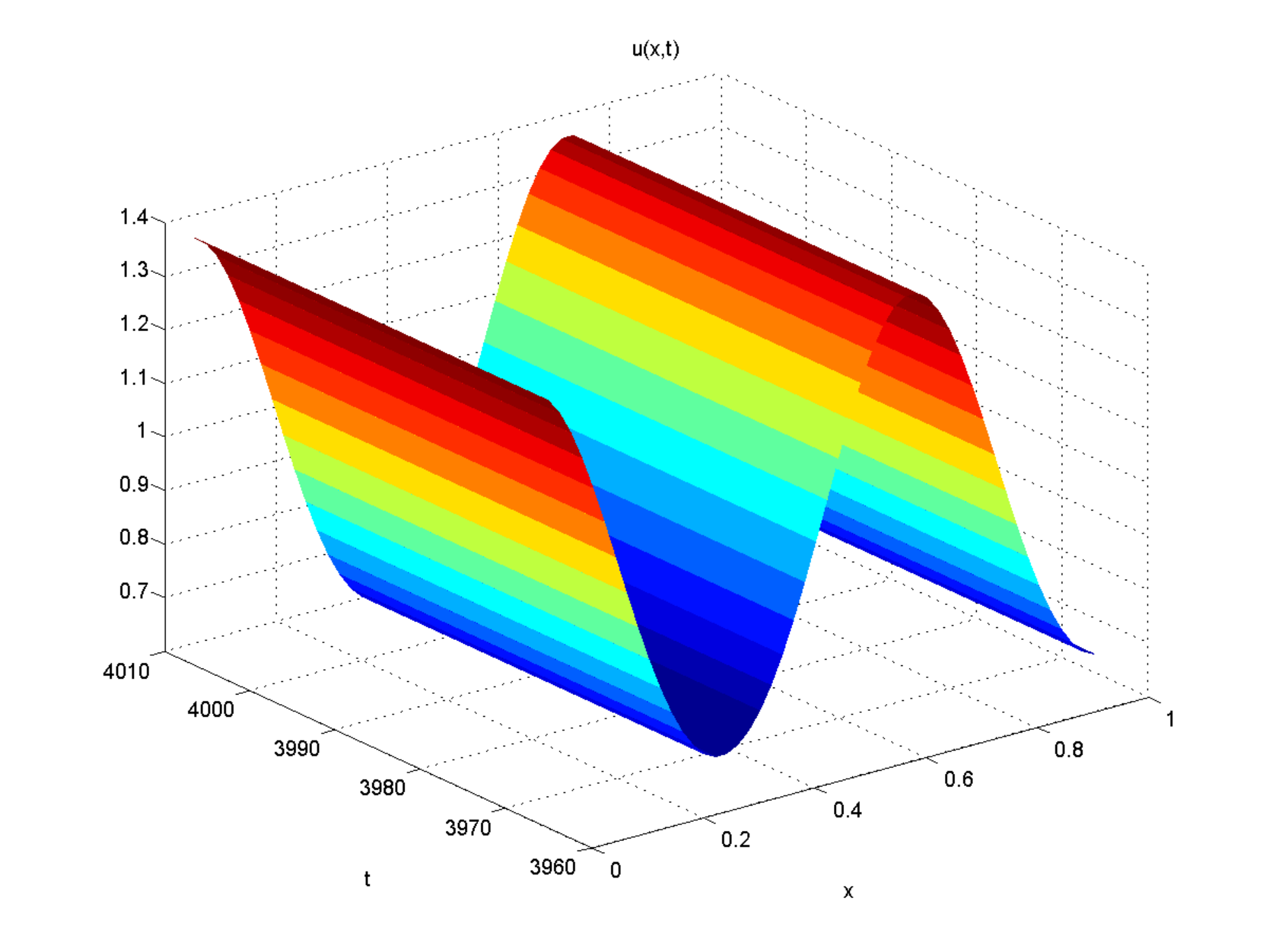}
	\end{minipage}}
	\subfigure[$v(t,x)$]{
		\begin{minipage}{0.48\linewidth}
			\centering
             \includegraphics[scale=0.42]{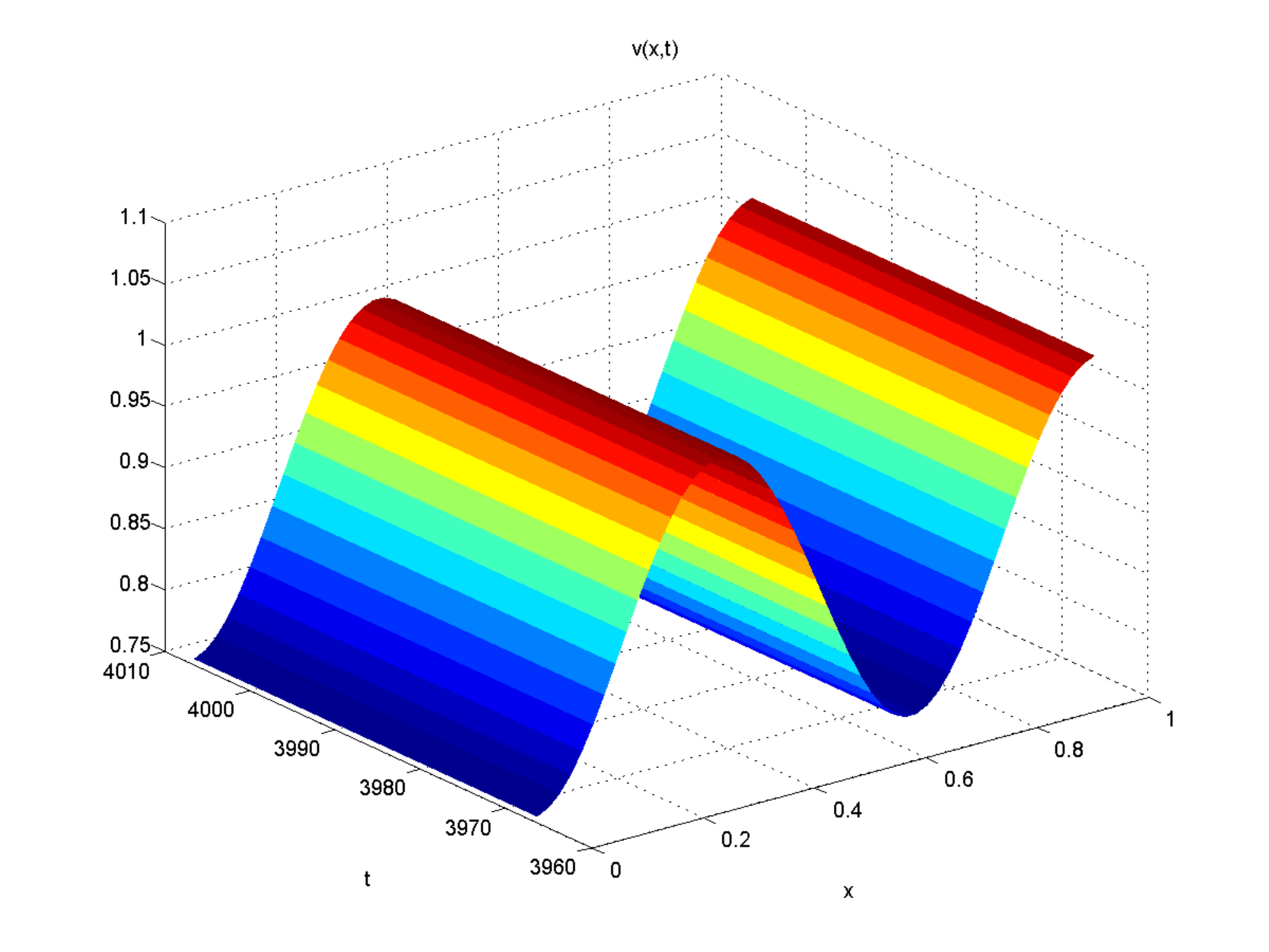}
		\end{minipage}
	}

	\subfigure[$u(t,x)$]{
		\begin{minipage}{0.48\linewidth}
			\centering
			\includegraphics[scale=0.42]{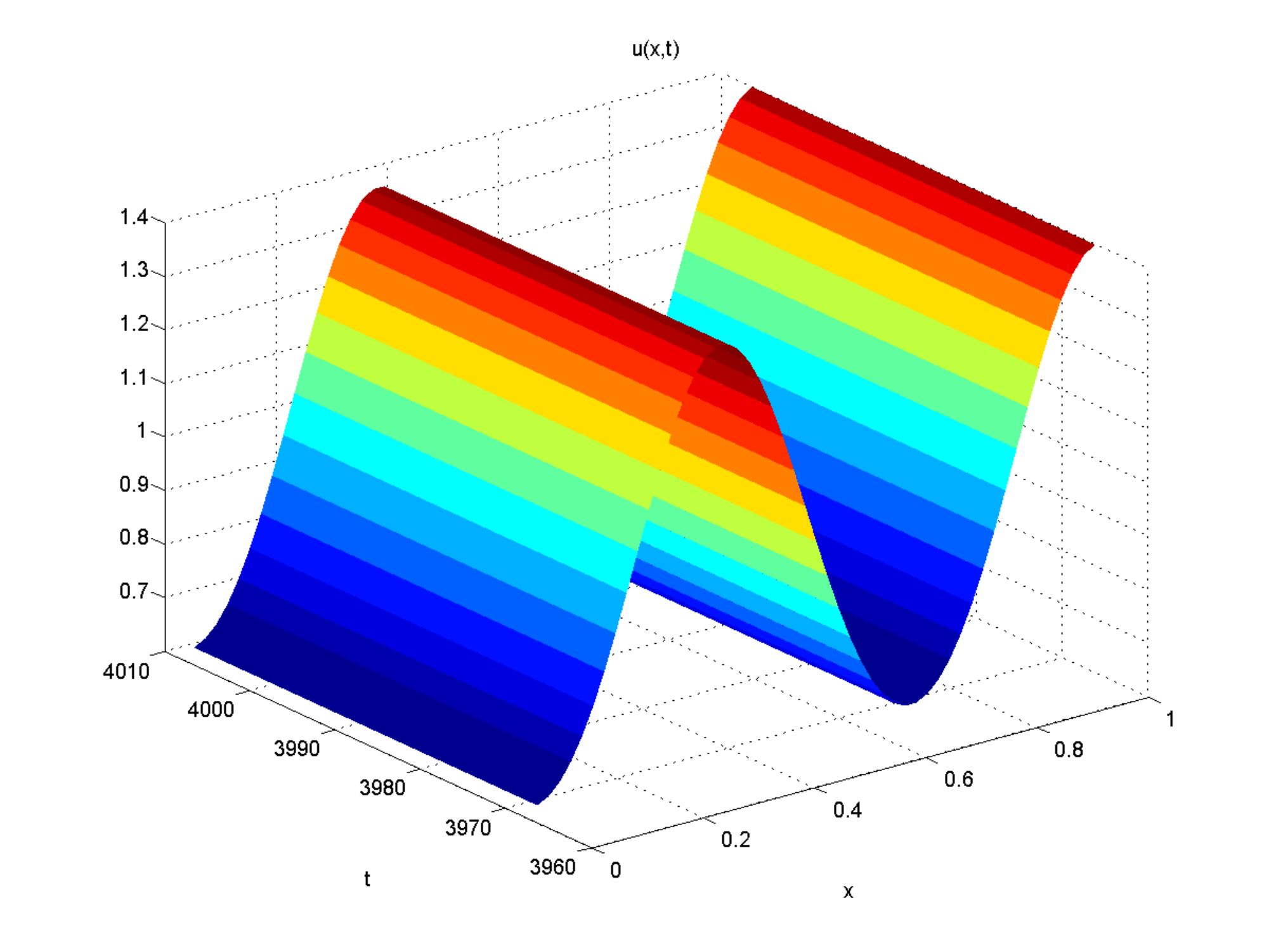}
	\end{minipage}}
	\subfigure[$v(t,x)$]{
		\begin{minipage}{0.48\linewidth}
			\centering
			\includegraphics[scale=0.42]{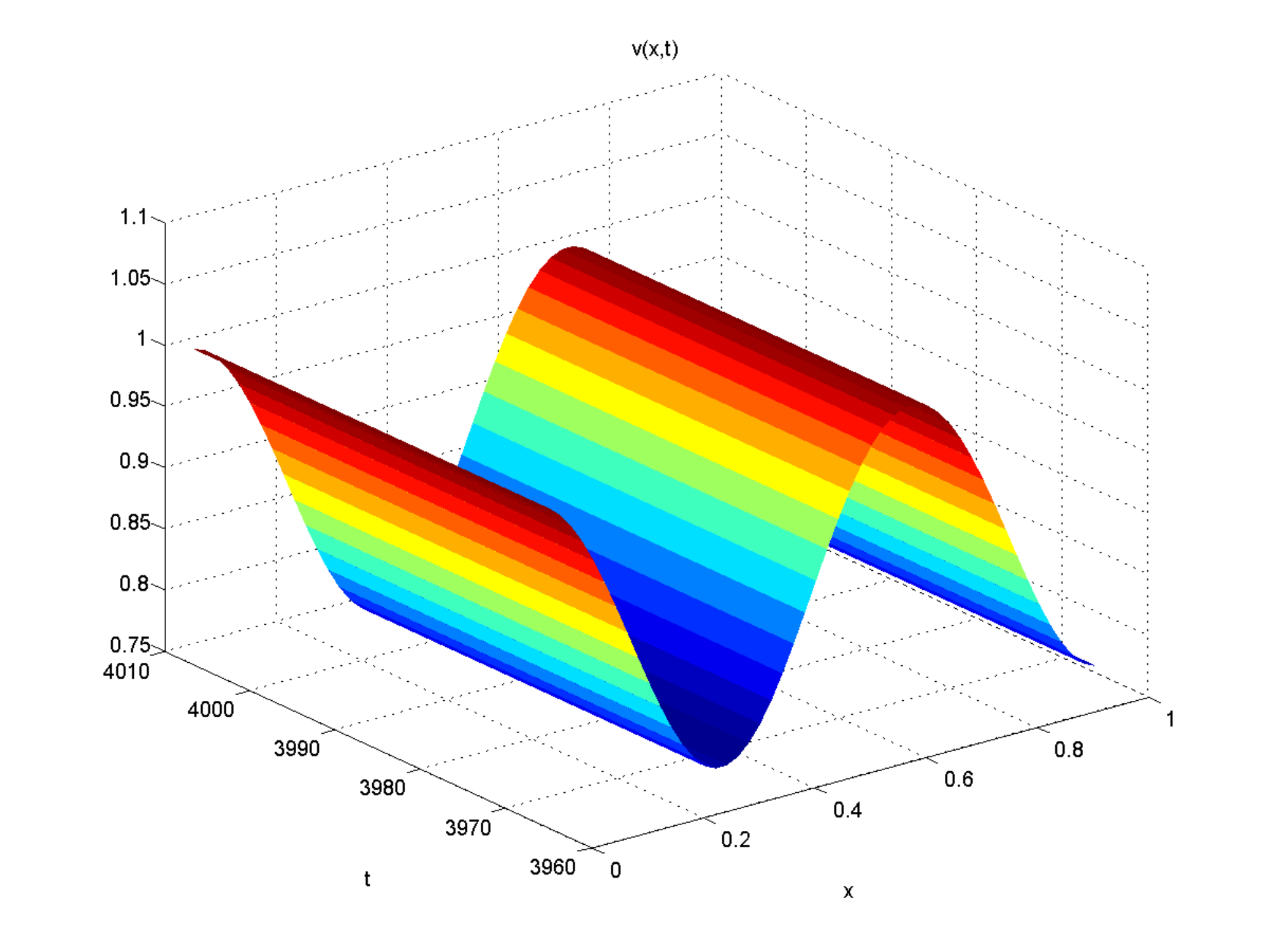}
		\end{minipage}
	}
	\caption{Two stable spatially inhomogeneous steady state solutions coexist for \eqref{3.1b} when  $(\tau,\ \varepsilon)\in D_5$.} \label{fig8}
\end{figure}

\end{enumerate}
\end{example}

\section{Conclusion}

Turing instability and Turing-Hopf bifurcation for a delayed reaction-diffusion Schnakenberg system are investigated, applying characteristic equation analysis, center manifold theorem and normal form method.

Firstly, on the basis of Yi\cite{YGLM}, a more larger range, where Turing instability doesn't occur, that is, the coexistence equilibrium is stable, has been provided, which is described as two sufficient conditions and one sufficient and necessary condition. They are independent of diffusion, dependent on diffusion but not dependent on wave numbers, and dependent on diffusion and wave numbers, respectively. In other words, we establish weaker conditions that guarantee the Turing  instability, of which two are necessary while one is sufficient and necessary.

Then, an explicit expression for the first Turing bifurcation curve has been obtained, on which the corresponding characteristic equation without delay has no root with positive real part. It is a piecewise smooth and continuous curve, of which piecewise points are exactly Turing-Turing bifurcation points. The expression explicitly depends on wave number $k$ and diffusion coefficient $d$, hence it is easy to find spatial pattern with arbitrary wave number. Based on this, the fact that spatially inhomogeneous steady state and spatially inhomogeneous periodic patterns with different spatial frequencies can be achieved via changing the diffusion rates, has been proven theoretically and shown numerically.

Furthermore, the joint effects of diffusion and delay ensure that Turing-Hopf bifurcation takes place. Normal forms truncated to order 3 restricted on center manifolds has been established, by utilizing the generic formulas (see \cite{JAS}), and all coefficients of normal forms are expressed explicitly, using the original system parameters $a,b,d,\varepsilon$ and the delay $\tau$. Bifurcation set on $(\tau,\varepsilon)$ parameters plane has been obtained. By discussing phase portraits, we have revealed that diffusion drives Turing bifurcation and leads to a pair of stable spatially inhomogeneous steady state solutions, while delay drives  Hopf bifurcation and leads to a stable periodic solutions. In addition, the joint effects of diffusion and delay can destabilize above solutions and generate a pair of stable spatially inhomogeneous periodic solutions.

Our results have indicated that when diffusion ratio $\varepsilon$ is relatively constant, diffusion coefficient $d$ of activator has great influence on the wave number $k$ (or the wave frequency) of spatial pattern. Smaller the diffusion coefficient $d$ is, larger the wave number $k$ is. Is the conclusion also suitable for other diffusion models?

Specifically, the phenomena observed by \cite{Gaffney2006Gene} that time delay can induce a failure of the Turing instability, have been theoretically explained, and the assertion drawn by Yi, Gaffney and Seirin-Lee (2017) that the modelling of gene expression time delays in Turing systems can eliminate or disrupt the formation of a stationary heterogeneous pattern in Schnakenberg system, has been further verified. These studies demonstrate that Turing-Hopf bifurcation is able to reveal the occurrence and development of some mixed spatiotemporal patterns, which may include both nonuniform spatially and temporally periodic patterns, and to further explain some complex biological phenomena. The research method in this paper can also be applied to any other delayed reaction-diffusion equation with similar degenerate critical points.

\section*{Acknowledgments}
The work was supported in part by the National Natural Science
Foundation of China (No.11371112).



\bibliographystyle{plain}
\bibliography{mybibfile}

\begin{thebibliography}{10}

\bibitem{AJ}
Q.~An and W.~Jiang.
\newblock Hopf-zero bifurcation and the normal forms in reaction-diffusion
  systems with time delays.
\newblock {\em preprint}, 2017.

\bibitem{Arafa2012Approximate}
A.~A.~M. Arafa, S.~Z. Rida, and H.~Mohamed.
\newblock Approximate analytical solutions of schnakenberg systems by homotopy
  analysis method.
\newblock {\em Appl. Math. Model.}, 36(10):4789--4796, 2012.

\bibitem{ChenS2012}
S.~Chen and J.~Shi.
\newblock Stability and hopf bifurcation in a diffusive logistic population
  model with nonlocal delay effect.
\newblock {\em J. Differential Equations}, 253(12):3440--3470, 2012.

\bibitem{Chen2016Stability}
Shanshan Chen and Jianshe Yu.
\newblock Stability analysis of a reaction-diffusion equation with
  spatiotemporal delay and dirichlet boundary condition.
\newblock {\em J. Dynam. Differential Equations}, 28(3-4):857--866, 2016.

\bibitem{Faria2000}
T.~Faria.
\newblock Normal forms and {Hopf} bifurcation for partial differential
  equations with delays.
\newblock {\em Trans. Amer. Math. Soc.}, 352(5):2217--2238, 2000.

\bibitem{Faria2002}
T.~Faria, W.~Huang, and J.~Wu.
\newblock Smoothness of center manifolds for maps and formal adjoints for
  semilinear fdes in general banach spaces.
\newblock {\em SIAM J. Math. Anal.}, 34(1):173--203, 2002.

\bibitem{Gaffney2006Gene}
E.~A. Gaffney and N.~A. Monk.
\newblock Gene expression time delays and turing pattern formation systems.
\newblock {\em Bull. Math. Biol.}, 68(1):99--130, 2006.

\bibitem{Guck1983}
J.~Guckenheimer and P.~Holmes.
\newblock {\em Nonlinear Oscillations, Dynamical Systems, and Bifurcations of
  Vector Fields}.
\newblock Springer-Verlag, 1983.

\bibitem{Guo2015}
S.~Guo.
\newblock Stability and bifurcation in a reaction-diffusion model with nonlocal
  delay effect.
\newblock {\em J. Differential Equations}, 259(4):1409--1448, 2015.

\bibitem{Gurdon2001Morphogen}
J.~B. Gurdon and P.~Y. Bourillot.
\newblock Morphogen gradient interpretation.
\newblock {\em Nature}, 413(6858):797--803, 2001.

\bibitem{Hadeler2012Interaction}
K.~P. Hadeler and S.~Ruan.
\newblock Interaction of diffusion and delay.
\newblock {\em Discrete Contin. Dyn. Syst. Ser. B}, 8(1):95--105, 2012.

\bibitem{Hale1977}
J.~K. Hale.
\newblock {\em Theory of Functional Differential Equations}.
\newblock Springer-Verlag, 1977.

\bibitem{Jang2004Global}
J.~Jang, W.~M. Ni, and M.~Tang.
\newblock Global bifurcation and structure of turing patterns in the 1-d
  lengyel-epstein model.
\newblock {\em J. Dynam. Differential Equations}, 16(2):297--320, 2004.

\bibitem{JAS}
W.~Jiang, Q.~An, and J.~Shi.
\newblock Formulation of the normal forms of turing-hopf bifurcation in
  reaction-diffusion systems with time delay.
\newblock {\em preprint}, 2017.

\bibitem{Just2001Spatiotemporal}
W.~Just, M.~Bose, S.~Bose, H.~Engel, and E.~Sch$\ddot{o}$ll.
\newblock Spatiotemporal dynamics near a supercritical turing-hopf bifurcation
  in a two-dimensional reaction-diffusion system.
\newblock {\em Phys. Rev. E}, 64(2):026219, 2001.

\bibitem{Kidachi1980On}
H.~Kidachi.
\newblock On mode interactions in reaction diffusion equation with nearly
  degenerate bifurcations.
\newblock {\em Prog. Theor. Phys.}, 63(4):1152--1169, 1980.

\bibitem{Lewis2003Autoinhibition}
J.~Lewis.
\newblock Autoinhibition with transcriptional delay: a simple mechanism for the
  zebrafish somitogenesis oscillator.
\newblock {\em Current Biology Cb}, 13(16):1398--408, 2003.

\bibitem{Li2013Hopf}
X.~Li, W.~Jiang, and J.~Shi.
\newblock Hopf bifurcation and turing instability in the reaction-diffusion
  holling-tanner predator-prey model.
\newblock {\em IMA J. Appl. Math.}, 78(2):287--306, 2013.

\bibitem{Murray2015}
J.~D. Murray.
\newblock {\em Mathematical Biology}.
\newblock Springer New York, 2003.

\bibitem{Ni2005Turing}
W.~M. Ni and M.~Tang.
\newblock Turing patterns in the lengyel-epstein system for the cima reaction.
\newblock {\em Trans. Amer. Math. Soc.}, 357(10):3953--3969, 2005.

\bibitem{YiFQ2013}
R.~Peng, F.~Yi, and X.~Zhao.
\newblock Spatiotemporal patterns in a reaction-diffusion model with the
  degnarrison reaction scheme.
\newblock {\em J. Differential Equations}, 254(6):2465--2498, 2013.

\bibitem{Ricard2009Turing}
M.~R. Ricard and S.~Mischler.
\newblock Turing instabilities at hopf bifurcation.
\newblock {\em J. Nonlinear Sci.}, 19(5):467--496, 2009.

\bibitem{Schnakenberg1979Simple}
J.~Schnakenberg.
\newblock Simple chemical reaction systems with limit cycle behaviour.
\newblock {\em J. Theor. Biol.}, 81(3):389, 1979.

\bibitem{SeirinLee2010}
L.~S. Seirin and E.~A. Gaffney.
\newblock Aberrant behaviours of reaction diffusion self-organisation models on
  growing domains in the presence of gene expression time delay.
\newblock {\em Bull. Math. Biol.}, 72(8):2161--2179, 2010.

\bibitem{Ruan2015}
H.~Shi and S.~Ruan.
\newblock Spatial, temporal and spatiotemporal patterns of diffusive
  predator-prey models with mutual interference.
\newblock {\em IMA J. Appl. Math.}, 80(5), 2015.

\bibitem{SongY2016}
Y.~Song, T.~Zhang, and Y.~Peng.
\newblock {Turing-Hopf} bifurcation in the reaction-diffusion equations and its
  applications.
\newblock {\em Commun. Nonlinear Sci. Numer. Simul.}, 33:229--258, 2016.

\bibitem{Su2010}
Y.~Su, J.~Wei, and J.~Shi.
\newblock Bifurcation analysis in a delayed diffusive nicholson blowflies
  equation.
\newblock {\em Nonlinear Anal. Real World Appl.}, 11(3):1692--1703, 2010.

\bibitem{Turing1952}
A.~M. Turing.
\newblock The chemical basis of morphogenesis.
\newblock {\em Philos. Trans. Roy. Soc. London Ser. B}, 237(641):37--72, 1952.

\bibitem{Veflingstad2005Effect}
S.~R. Veflingstad, E.~Plahte, and N.~A.~M. Monk.
\newblock Effect of time delay on pattern formation: Competition between
  homogenisation and patterning.
\newblock {\em Phys. D}, 207(3-4):254--271, 2005.

\bibitem{Wang2016Spatiotemporal}
Jinfeng Wang.
\newblock Spatiotemporal patterns of a homogeneous diffusive predator-prey
  system with holling type iii functional response.
\newblock {\em J. Dynam. Differential Equations}, pages 1--27, 2016.

\bibitem{Ward2002The}
Michael~J. Ward and Juncheng Wei.
\newblock The existence and stability of asymmetric spike patterns for the
  schnakenberg model.
\newblock {\em Studies in Applied Mathematics}, 109(3):229¨C264, 2002.

\bibitem{Holmes1997}
R.~W. Wittenberg and P.~Holmes.
\newblock The limited effectiveness of normal forms: A critical review and
  extension of local bifurcation studies of the brusselator pde.
\newblock {\em Phys. D}, 100(1-2):1--40, 1997.

\bibitem{YuanXP2010}
X.~Yan and W.~Li.
\newblock Stability of bifurcating periodic solutions in a delayed
  reaction-diffusion population model.
\newblock {\em Nonlinearity}, 23(6):1413, 2010.

\bibitem{YGLM}
F.~Yi, E.~A. Gaffney, and L.~S. Seirin.
\newblock The bifurcation analysis of turing pattern formation induced by delay
  and diffusion in the schnakenberg system.
\newblock {\em Discrete Contin. Dyn. Syst. Ser. B}, 22(2):647--668, 2017.

\bibitem{Yi2009}
F.~Yi, J.~Wei, and J.~Shi.
\newblock Bifurcation and spatiotemporal patterns in a homogeneous diffusive
  predator-prey system.
\newblock {\em J. Differential Equations}, 246(5):1944--1977, 2009.

\end{thebibliography}

\end{document}